\documentclass[11pt,reqno,final]{amsart}

\usepackage{amsfonts}
\usepackage[notcite,notref]{showkeys}
\usepackage{url,hyperref,multirow}
\usepackage{color}
\usepackage{cases}
\usepackage{subfigure}
\usepackage[ruled,vlined]{algorithm2e}
\usepackage{algorithmic}
\usepackage{subeqnarray}
\usepackage{amsmath}

\allowdisplaybreaks[4]
\usepackage{epsfig,amsmath,bm,caption2,epstopdf}
\usepackage{amssymb,version,graphicx,fancybox,mathrsfs,pifont,booktabs}

\catcode`\@=11 \theoremstyle{plain}
\@addtoreset{equation}{section}   

\@addtoreset{figure}{section}
\renewcommand\thefigure{\thesection.\@arabic\c@figure}
\@addtoreset{table}{section}
\renewcommand\thetable{\thesection.\@arabic\c@table}

\newtheorem{thm}{\bf Theorem}[section]
\newtheorem{cor}{\bf Corollary}[section]
\newtheorem{prop}{Proposition}[section]

\newtheorem{lmm}{\bf Lemma}[section]

\newenvironment{lemma}{\begin{lmm}}{\end{lmm}}
\theoremstyle{remark}

\theoremstyle{definition}
\newtheorem{defn}{Definition}[section]

\newtheorem{exm}{\bf Example} 

\newcommand{\bs}[1]{\boldsymbol{#1}}

\newcommand \dint {\displaystyle \int}

\newcommand \R {\mathbb{R}}

\begin{document}

\title[An efficient method for fractional reaction-diffusion equations]{An efficient spectral-Galerkin method for fractional reaction-diffusion equations in unbounded domains}
\author[H. Yuan]{ Huifang Yuan}
\thanks{SUSTech  International Center for Mathematics,  Southern University of Science and Technology, Shenzhen, 518055, China, and School of Mathematics and Statistics, Wuhan University, Wuhan, 430072, China.  Email:  yuanhf@sustech.edu.cn (H. Yuan).}



 \begin{abstract}
 In this work, we apply a fast and accurate numerical method for solving fractional reaction-diffusion equations in unbounded domains. By using the Fourier-like spectral approach in space, this method can effectively handle the fractional Laplace operator, leading to a fully diagonal representation of the fractional Laplacian.  To fully discretize the underlying nonlinear reaction-diffusion systems, we propose to use an accurate time marching scheme based on ETDRK4.  Numerical examples are presented to illustrate the effectiveness of the proposed method.
\end{abstract}

\maketitle

\vspace*{-10pt}
\section{Introduction}
Progress in the last few decades has shown that many complex systems in science, economics and engineering are found to be more accurately described by fractional differential equations, see \cite{bouchaud1990anomalous,cartea2007fluid,metzler2000random,west1994fractal}. Usually, in these equations fractional derivatives are involved to deal with long range time dependence or long range spatial interactions. Among them, typical fractional operators are Reimann-Liouville fractional derivatives, Caputo fractional derivatives, Riesz fractional derivatives, fractional Laplacian and so on. Since analytic solutions to fractional differential equations are usually unknown, there has been a growing interest in developing effective and efficient numerical methods, see \cite{alikhanov2015new,bonito2019numerical,chen2016generalized,deng2009finite,lin2007finite,tian2015class,zayernouri2013fractional} and the references therein.

In this paper, we are going to deal with the fractional reaction-diffusion equations in unbounded domains where the second order space derivative is replaced with fractional Laplacian. To better illustrate our idea, we consider the following model equation
\begin{equation}\label{frde}
\begin{cases}
&\frac{\partial u}{\partial t}=-D(-\Delta)^{s}u+f(u), \quad t\in(0,T],\;x\in\mathbb{R}^{d},\\
&u(x,t)\to 0, \qquad\qquad\qquad\;\;\,|x|\to\infty,\\
&u(x,0)=u_{0}(x),\qquad\qquad\quad x\in\mathbb{R}^{d},
\end{cases}
\end{equation}
for $s\in (0,1.5)$, where $D>0$ is the diffusion coefficient and $f(u)$ is the reaction term. Reaction-diffusion equations model the change in space and time of the concentration of substances under the influence reaction and diffusion, and have found wide applications in biology, physics, chemistry and engineering, see, for example, \cite{cantrell2004spatial,klein2006pollen,pouchol2019phase}. Examples of particular interest include the Fisher-Kolmogorov equation when $f(u)=ru(1-\frac{u}{K})$, see \cite{owolabi2016solution,zhu2017numerical}, and the Allen-Cahn equation when $f(u)=u-u^3$, see \cite{hou2017numerical,liao2020energy,tang2016implicit}. For numerical methods, readers may see \cite{baeumer2008numerical,owolabi2018efficient,pindza2016fourier} and the references therein. For integer order reaction-diffusion equation, the use of integer order derivatives lies in the assumption that the random motion is stochastic Gaussian process. However, more and more studies has shown that certain systems exhibit anomalous diffusion, see, e.g., \cite{del2003front,seki2003fractional}, and the references therein. In this paper, we study the case when the second-order space derivative is replaced with the fractional Laplace operator, which can account for anomalous diffusion phenomenon by allowing arbitrary jump length distribution, including the case with infinite variance. For time-fractional reaction-diffusion models, readers may refer to \cite{atangana2016new,gafiychuk2008mathematical,henry2006anomalous,seki2003fractional}.

Fractional Laplacian can be defined as a pseudo-differential operator via  the Fourier transform for $s>0$:
\begin{equation}\label{Ftransform}
\begin{split}
 (-\Delta)^s u(x):={\mathscr F}^{-1}\big[|\xi|^{2s}  {\mathscr F}[u](\xi)\big](x), \quad  \forall\, x \in {\mathbb R}^d.
\end{split}
\end{equation}
Or equivalently, it can be defined for $s\in(0,1)$ by singular representation (cf. \cite[Prop. 3.3]{Nezza2012BSM}):
\begin{equation}\label{fracLap-defn}
(-\Delta)^s u(x)=C_{d,s}\, {\rm p.v.}\! \int_{\mathbb R^d} \frac{u(x)-u(y)}{|x-y|^{d+2s}}\, {\rm d}y,\quad x\in {\mathbb R}^d,
\end{equation}
where ``p.v." stands for the principle value and the normalisation constant is given as
\begin{equation}\label{Cds}
C_{d,s}:=\Big(\int_{\mathbb R^d} \frac{1-\cos \xi_1}{|\xi|^{d+2s}}\,{\rm d}\xi\Big)^{-1}=
\frac{2^{2s}s\Gamma(s+d/2)}{\pi^{d/2}\Gamma(1-s)}.
\end{equation}
Various numerical treatment of the fractional Laplace operator has been proposed, see, e.g., finite difference method\cite{duo2018novel,duo2019accurate,minden2020simple,tian2015class}, finite element method\cite{acosta2017short,bonito2019numerical,chen2018multigrid}, and spectral method \cite{mao2017hermite,tang2020rational,tang2018hermite}. Recently, in \cite{sheng2020fast}, biorthogonal mapped Chebyshev functions are proposed as the basis for fractional differential equations in unbounded domain, which then lead to diagonal stiffness matrix by using the Dunford-Taylor formula. This provides an efficient tool for numerical approximations of fractional reaction-diffusion equations, especially in multidimensional problems. In this paper, we will follow the routines in \cite{sheng2020fast} to deal with the fractional Laplace operator.

The aim of this work is to propose a reliable and efficient approach to solve the fractional reaction-diffusion equation in unbounded domains involving the fractional Laplace operator. In order to obtain higher accuracy, we adapt the fourth-order exponential time differencing Runge-Kutta (ETDRK4) method for time stepping, which is based on the exact integration of the linear part and the approximation of an integral involving the nonlinear part.

The main contributions of this work are three folds:
\begin{itemize}
\item[(1)] Since fractional derivatives are usually nonlocal, here we study the unbounded domain directly, different from other previous works where domain truncation is usually used.
\item[(2)] The biorthogonal mapped Chebyshev functions are used as the basis for our spectral-Galerkin method, which lead to diagonal stiffness matrix, and thus very efficient even in higher dimensions.
\item[(3)] The ETDRK4 time stepping method is used for the resulted semidiscrete problem which is 4th-order accurate in time. This ensures that numerical treatment of the fractional reaction-diffusion equation is both more efficient and accurate.
\end{itemize}

The rest of this paper is organized as follows. In the next section, we briefly give some preliminaries on the properties of biorthogonal mapped Chebyshev functions and the Dunford-Taylor formula for $s\in(0,2)$. The numerical approach based on spectral-Galerkin method and ETDRK4 time stepping method for fractional reaction-diffusion equation is developed in section 3, where the linear stability of ETDRK4 is also analysed. In section 4, various numerical examples are provided to illustrate the effectiveness of proposed method, including the Fisher-Kolmogorov equation, Allen-Cahn equation, Gray-Scott equation and FitzHugh-Nagumo equation. The final section is for some concluding remarks.


\section{Preliminary}

The mapped Chebyshev functions (MCFs) based on the Fourier-like biorthogonal property was proposed and studied recently in \cite{sheng2020fast}. The biorthogonal property enables us to make the fractional Laplacian fully diagonalised, and the
complexity of solving an elliptic fractional PDE is quasi-optimal, i.e., $O((N \log2 N)^
d)$ with $N$ being the
number of modes in each spatial direction. The present section is to provide some basic preliminaries
useful for the efficient spectral algorithms to be provided
in the forthcoming section.

\subsection{Fourier-like mapped Chebyshev functions}
Let  $T_n(y)=\cos(n\, {\rm arccos}(y)),$ $y\in \Lambda:=(-1,1)$  be the Chebyshev polynomial of degree $n$. The mapped Chebyshev functions (MCFs) are defined as in \cite{ben2002modified,shen2011spectral,shen2014approximations}.
\begin{defn}\label{defnMCF} {\em
Introduce the one-to-one algebraic mapping 
\begin{equation}\label{AlgeMapp}
x=\frac{y}{\sqrt{1-y^2}},\quad y=\frac{x}{\sqrt{1+x^2}},\quad x\in {\mathbb R},\;\; y\in \Lambda,
\end{equation}
and  define the MCFs as
 \begin{equation}\label{defiOMCF}
\mathbb{T}_n(x)= \frac{1} {\sqrt{c_n\pi/2}} \sqrt{1-y^2}\, T_n(y)= \frac{1}{\sqrt{c_n\pi/2}} \frac 1{\sqrt{1+x^2}}\, T_n\Big(\frac x {\sqrt{1+x^2}}\Big),
\end{equation}
for $x\in \mathbb R$ and integer $n\ge 0.$}
\end{defn}
We have the following important property of the MCFs with respect to the mass matrix and stiffness matrix, see \cite[Proposition 2.4]{shen2014approximations}.
\begin{prop}\label{orthMCF} The MCFs are orthonormal in $L^2(\mathbb R),$
and we have for the stiffness matrix
\begin{equation}\label{derivOrth}
\begin{split}
&{S}_{mn}={S}_{nm}= \dint_{\R}\mathbb{T}'_n(x)\,\mathbb{T}'_m(x)\,{\rm d}x\\
&=\begin{cases}
\displaystyle \frac{1}{c_n}\Big(\frac{(4c_{n-1}-c_{n-2})(n-1)^2}{16}+\frac{(4c_{n+1}-c_{n+2})(n+1)^2}{16}-\frac{c_{n}}{4}\Big), & {\rm if}\;\; m=n, \\[6pt]
\displaystyle \frac{1}{\sqrt{c_nc_{n+2}}}\Big(\frac{(c_{n}-c_{n+2})(n+1)}{8}-\frac{c_{n+1}(n+1)^2}{4}\Big), & {\rm if}\;\; m=n+2,\\[6pt]
\displaystyle \frac{1}{\sqrt{c_nc_{n+4}}}\Big(\frac{c_{n+2}(n+1)(n+3)}{16}\Big),  & {\rm if}\;\; m=n+ 4,\\
\displaystyle 0, &\text{otherwise}.
\end{cases}
\end{split}
\end{equation}
\end{prop}
The Fourier-like biorthogonal  MCFs $\{\widehat{\mathbb T}_n\}^N_{p=0}$ are defined as a linear combination of the mapped Chebyshev functions, where the expansion coefficients are the matrix formed by the orthonormal eigenvectors of $\bs{S}$. Readers may refer to \cite{sheng2020fast} for more details. Then they satisfy the following property:
\begin{lemma}\label{partial-diagonal}
$\{\widehat{\mathbb T}_n\}^N_{n=0}$  form an equivalent basis with $\{{\mathbb T}_n\}^N_{p=0}$, denoted as $\mathbb{V}_{\!N},$  and they are biorthogonal in the sense that
\begin{equation}\label{bi-orth-eqn}
(\widehat{\mathbb T}_m,\widehat{\mathbb T}_n)_{L^2(\mathbb{R})}=\delta_{mn},\quad  \big(\,\widehat{\mathbb T}_m', \widehat{\mathbb T}_n'\,\big)_{L^2(\mathbb{R})}=\lambda_n\delta_{mn},
\quad 0\le m,n\le N,
\end{equation}
where $\delta_{mn}$ is the Kronecker symbol and $\lambda_{n}$,\;$n=0,1,\cdots,N$, are the eigenvalues of $\bs{S}$.
\end{lemma}

For $d-$dimensional sobolev space, define 
\begin{equation}\label{VnVn}
{\mathbb V}_{\!N}^d=\mathbb V_{\!N}\otimes\cdots \otimes \mathbb V_{\!N},
\end{equation}
which is the tensor product of $d$ copies of  $\mathbb V_{\!N}$. Define the $d$-dimensional tensorial Fourier-like basis and denote the vector of the corresponding eigenvalues in \eqref{bi-orth-eqn} by
\begin{equation}\label{mathTT}
 \widehat {\mathbb T}_n(x)=\displaystyle\prod^d_{j=1} \widehat {\mathbb T}_{n_j}(x_j), \quad x\in\mathbb{R}^d;\quad
\lambda_n=(\lambda_{n_1},\cdots,\lambda_{n_d})^t .
\end{equation}
Accordingly,  we have 
\begin{equation}
\label{apprspacRd}
\mathbb{V}_{\!N}^d={\rm span}\big\{ \widehat {\mathbb T}_n(x),~n\in\Upsilon_{\!N}\big\},
\end{equation}
where the index set
\begin{equation}\label{UpsilonA}
  \Upsilon_{\!N}:=\big\{n=(n_1,\cdots, n_d)\,:\, 0\le n_j\le N,\; 1\le j\le d\big\}.
\end{equation}

As an extension of Lemma \ref{partial-diagonal}, the following attractive property of the tensorial Fourier-like MCFs is obtained.
\begin{lmm}  For the tensorial Fourier-like MCFs defined in \eqref{mathTT}, we have
\begin{equation}\label{pqN}
( \widehat {\mathbb T}_m,  \widehat {\mathbb T}_n)_{L^2(\mathbb{R}^d)}=\delta_{mn}\,;  \quad
\big(\nabla \widehat {\mathbb T}_m,  \nabla \widehat {\mathbb T}_n\big)_{L^2(\mathbb{R}^d)}=|\lambda_n|_1\,\delta_{mn},
\end{equation}
where $|\cdot|_{1}$ denotes the $L^{1}$ norm.
\end{lmm}

\subsection{Fractional Sobolev space}
For real $s\ge 0,$ define the fractional Sobolev space (cf. \cite[P. 530]{Nezza2012BSM}):
\begin{align}\label{Hssps}
H^s(\mathbb{R}^d)=\Big\{u\in L^{2}(\mathbb {R}^d):\,  \lVert u\rVert_{H^s(\mathbb{R}^d)}^2=\int_{\mathbb R^d} (1+\lvert\xi\rvert^{2s})
\big|\mathscr{F}[u](\xi)\big|^{2}{\rm d}\xi<+\infty\Big\},
\end{align}
\subsection{Dunford-Taylor formula}
Besides the biorthogonal MCFs, the other key component of the spectral-Galerkin method in \cite{sheng2020fast} is the following Dunford-Taylor formula for fractional Laplacian.
\begin{lemma}\label{DTthm}
 For any $u,v\in H^{s}(\mathbb{R}^d)$ with $s\in (0,1)$, we have
\begin{equation}\label{DTfor}
\left((-\Delta)^{\frac{s}2} u,(-\Delta)^{\frac{s}2} v\right)_{L^2(\mathbb{R}^d)}=C_s \int_0^\infty t^{1-2s} \int_{\mathbb R^d}
 \big((-\Delta)(\mathbb I-t^2 \Delta)^{-1} u \big)(x)\,  v(x)\, {\rm d}x\,{\rm d}t,
\end{equation}
where $\mathbb I$ is the identity operator and
\begin{equation}
C_s=\frac{2 \sin (\pi s)}{\pi}.
\end{equation}
\end{lemma}
Proof of the Dunford-Taylor formula can be found in \cite{bonito2019numerical}. In fact, similar results can be obtained for $s\in(1,2)$ by dividing $s$ into $s_{1}+s_{2}$, where $s_{1},s_{2}\in (0,1)$.
\begin{thm}\label{DTs2thm}
 For any $u\in H^{2s_{1}}(\mathbb{R}^{d})$, $v\in H^{2s_{2}}(\mathbb{R}^{d})$ with $s_{1},s_{2}\in (0,1)$, we have
\begin{equation}\label{DTfor}
\begin{split}
&\big((-\Delta)^{s_{1}} u,(-\Delta)^{s_{2}} v\big)=C_{s_{1}}C_{s_{2}} \int_0^\infty t_{1}^{1-2s_{1}} \int_0^\infty t_{2}^{1-2s_{2}}\\
&\qquad\qquad \Big((-\Delta)(\mathbb I-t_{1}^2 \Delta)^{-1} u(x),(-\Delta)(\mathbb I-t_{2}^2 \Delta)^{-1} v(x)\Big) {\rm d}t_{1}\,{\rm d}t_{2}.
 \end{split}
\end{equation}
\end{thm}
\begin{proof}
By Parseval's identity, we have
\begin{equation}\label{FTs2}
\big((-\Delta)^{s_{1}} u,(-\Delta)^{s_{2}} v\big)=\Big(|\xi|^{2s_{1}}{\mathscr F}[u](\xi),\overline{|\xi|^{2s_{2}}{\mathscr F}[v](\xi)}\Big).
\end{equation}
Note that
\begin{equation}\label{ytxi00}
{|\xi|^{2\alpha}}=\frac  {2\sin \pi \alpha} \pi \int_{0}^\infty  \frac{ |\xi|^2\, t^{1-2\alpha}}{1+t^2 |\xi|^2} {\rm d}t,\quad \alpha\in(0,1).
\end{equation}
Then we have
\begin{equation*}
\begin{split}
&\quad\big((-\Delta)^{s_{1}} u,(-\Delta)^{s_{2}}v\big)\!=C_{s_{1}}C_{s_{2}}\!\!\int_{\mathbb{R}^{d}}\!\int_{0}^{\infty}\!\frac{|\xi|^{2}t_{1}^{1-2s_{1}}}{1+t_{1}^{2}|\xi|^2}{\rm d}t_{1}{\mathscr F}[u](\xi)\!\!\int_{0}^{\infty}\!\frac{|\xi|^{2}t_{2}^{1-2s_{2}}}{1+t_{2}^{2}|\xi|^2}{\rm d}t_{2}\overline{{\mathscr F}[v](\xi)}{\rm d}\xi\\
&=C_{s_{1}}C_{s_{2}}\int_{0}^{\infty}t_{1}^{1-2s_{1}}\int_{0}^{\infty}t_{2}^{1-2s_{2}}\Big(\frac{|\xi|^{2}}{1+t_{1}^{2}|\xi|^2}{\mathscr F}[u](\xi),\frac{|\xi|^{2}}{1+t_{2}^{2}|\xi|^2}\overline{{\mathscr F}[v](\xi)}\Big){\rm d}t_{1}{\rm d}t_{2}\\
&=C_{s_{1}}C_{s_{2}}\int_{0}^{\infty}t_{1}^{1-2s_{1}}\int_{0}^{\infty}t_{2}^{1-2s_{2}}\Big((-\Delta)(\mathbb I-t_{1}^2 \Delta)^{-1} u(x),(-\Delta)(\mathbb I-t_{2}^2 \Delta)^{-1} v(x)\Big){\rm d}t_{1}{\rm d}t_{2}.
\end{split}
\end{equation*}
\end{proof}
\section{Spectral-Galerkin method}
In this section, we conduct the numerical approximations for \eqref{frde}, where we first formulate the spectral-Galerkin method to deal with the space fractional derivatives, and later advance the resulting ODE in time with ETDRK4.
\subsection{Space discretisation}
A weak form of \eqref{frde} is to find $u\in H^s(\mathbb{R}^d)$ such that
\begin{equation}\label{weakform}
\Big(\frac{\partial u}{\partial t},v\Big)=-D\big((-\Delta)^{s/2}u,(-\Delta)^{s/2}v\big)+\big(f(u),v\big), \quad \forall v\in H^s(\mathbb{R}^d).
\end{equation}
Then using the Dunford-Taylor formula in Lmm.\,\ref{DTthm} for $s\in(0,1)$, our spectral-Galerkin method is to find $u_{N}\in\mathbb{V}_{N}^{d}$ such that
\begin{equation}\label{galerkin}
\Big(\frac{\partial u_{N}}{\partial t},v_{N}\Big)=-D\,C_s \int_0^\infty t^{-1-2s} (u_N-w_N, v_N){\rm d}t+\big(I_{N}f(u_{N}),v_{N}\big), \quad \forall v_{N}\in\mathbb{V}_{N}^{d},
\end{equation}
where we find $w_N:=w_N(u_N,t)\in \mathbb{V}_{\!N}^d$ such that for any $t>0,$
\begin{equation}\label{bSeqn}
t^{2} (\nabla w_N,\nabla \psi)+(w_N, \psi)=(u_N, \psi), \quad \forall \psi \in \mathbb{V}_{\!N}^d.
\end{equation}
For $s\in(1,1.5)$, using the Dunford-Taylor formula in Thm \ref{DTs2thm}, the spectral-Galerkin method is to find  $u_N\in {\mathbb V}_{\!N}^d$ such that
\begin{equation}\label{auveqn23}
\begin{split}
 \Big(\frac{\partial u_{N}}{\partial t},v_{N}\Big)&=-D\,C_{s/2}\,C_{s/2}\!\int_{0}^{\infty}\!t_{1}^{-1-s}\!\int_{0}^{\infty}\!t_{2}^{-1-s}(u_N-w_N, v_N-\rho_{N})\,  {\rm d}t_{1}{\rm d}t_{2}\\
 & \qquad+\big(I_{N}f(u_{N}),v_{N}\big),\quad \forall v_N\in \mathbb{V}_{\!N}^d,
 \end{split}
\end{equation}
where we find $w_N:=w_N(u_N,t)\in \mathbb{V}_{\!N}^d$ and $\rho_N:=\rho_N(v_N,t)\in \mathbb{V}_{\!N}^d$ such that for any $t>0,$
\begin{equation}\label{bSeqn}
\begin{split}
&t^{2} (\nabla w_N,\nabla \psi)+(w_N, \psi)=(u_N, \psi), \quad \forall \psi \in \mathbb{V}_{\!N}^d;\\
&t^{2} (\nabla \rho_N,\,\nabla \chi)\,+(\rho_N,\, \chi)=(v_N,\, \chi), \quad \forall \chi \in \mathbb{V}_{\!N}^d.
\end{split}
\end{equation}

Let $u_{N}(x,t)=\sum_{n\in\Upsilon_{N}}\widehat{u}_{n}(t)\widehat{\mathbb{T}}_{n}(x)$ and $v_{N}=\widehat{\mathbb{T}}_{n}(x)$, $n=(n_{1},n_{2},\cdots,n_{d})\in\Upsilon_{N}$, then similar to \cite[Thm 3.2]{sheng2020fast}, we have the following unified semi-discrete form for all $s\in(0,1.5)$
\begin{equation}\label{semidiscrete}
\frac{\partial \widehat{u}_{n}(t)}{\partial t}=-D|\lambda_{n}|_{1}^{s}\widehat{u}_{n}(t)+\widehat{f}_{n}(u_{N}), \quad n\in\Upsilon_{N},
\end{equation}
where
\begin{equation}
|\lambda_{n}|_{1}^{s}=(\lambda_{n_{1}}+\lambda_{n_{2}}+\cdots+\lambda_{n_{d}})^{s},\quad \widehat{f}_{n}(u_{N})=\big(I_{N}f(u_{N}),\widehat{\mathbb{T}}_{n}\big), \quad n\in\Upsilon_{N}.
\end{equation}
We remark that a diagonal representation of the fractional Laplace operator is obtained based on the biorthogonal MCFs in \eqref{mathTT} and the Dunford-Taylor formula in Lmm. \ref{DTthm} and Thm. \ref{DTs2thm} yields a unified semi-discrete form for all $s\in(0,1.5)$.
\subsection{Time discretisation}
We first write the resulted ODE system \eqref{semidiscrete} in the following more general form:
 \begin{equation}\label{semilinearparabolic}
{\mathcal U}_{t}=\mathbf{L}{\mathcal U}+\mathbf{N}({\mathcal U},t),
\end{equation}
where $\mathbf{L}$ and $\mathbf{N}$ denote linear and nonlinear terms, respectively. In our case, $\mathbf{L}$ is diagonal matrix with the components $-D|\lambda_{n}|_{1}^{s}$, and $\mathbf{N}({\mathcal U})$ represents the expansion coefficients of reaction term.

Typical numerical methods for solving problems of this kind include implicit-explicit method, integrating factor method and exponential time differencing (ETD) method, etc. To remove the stiffness caused by the linear part, we choose the ETD method which involves exact integration of the linear part followed by an approximation of the integral of the nonlinear term. Multiplying \eqref{semilinearparabolic} by the term $e^{-\mathbf{L}t}$ and integrating over a single time step $\tau$ give
\begin{equation}
{\mathcal U}_{n+1}=e^{\mathbf{L}\tau}{\mathcal U}_{n}+e^{\mathbf{L}\tau}\int_{0}^{\tau}e^{-\mathbf{L}t}\,\mathbf{N}\big({\mathcal U(t_{n}+t)},t_{n}+t\big)\,{\rm d}t.
\end{equation}

For the approximation of the integral involving the nonlinear term, Runge-Kutta method is taken since it is easy to implement and requires only one previous evaluation when compared to multistep methods. ETDRK4 method was first proposed in \cite{cox2002exponential}, and then modified in \cite{kassam2005fourth} to increase numerical stability. In this work, we utilize the fourth-order Runge-Kutta scheme of Krogstad in \cite{krogstad2005generalized}, denoted as ETDRK4-B, which has smaller local truncation error and larger stability properties. More precisely, for the semi-discrete ODE system \eqref{semilinearparabolic}, the ETDRK4 scheme is
\begin{equation*}
\begin{split}
{\mathcal U}_{n+1}=&\;e^{{\mathbf{L}}\tau}{\mathcal U}_{n}+\tau\big[4\varphi_{3}(\mathbf{L}\tau)-3\varphi_{2}(\mathbf{L}\tau)+\varphi_{1}(\mathbf{L}\tau)\big]{\mathbf{N}}({{\mathcal U}_{n},t_{n}})\\
&+2\tau\big[\varphi_{2}(\mathbf{L}\tau)-2\varphi_{3}(\mathbf{L}\tau)\big]\mathbf{N}({\mu_{2},t_{n}+\tau/2})\\
&+2\tau\big[\varphi_{2}(\mathbf{L}\tau)-2\varphi_{3}(\mathbf{L}\tau)\big]\mathbf{N}(\mu_{3},t_{n}+\tau/2)\\
&+\tau\big[4\varphi_{3}(\mathbf{L}\tau)-\varphi_{2}(\mathbf{L}\tau)\big]{\mathbf{N}}({\mu_{4},t_{n}+\tau}),
\end{split}
\end{equation*}
where $\tau$ is the time stepsize and the stages $\mu_{2},\,\mu_{3},\,\mu_{4}$ are defined as
\begin{equation*}
\begin{split}
\mu_{2}=&\;e^{{\mathbf{L}}\tau/2}{\mathcal U}_{n}+(\tau/2)\varphi_{1}(\mathbf{L}\tau/2){{\mathbf{N}}}({\mathcal U}_{n},t_{n}),\\
\mu_{3}=&\;e^{{\mathbf{L}}\tau/2}{\mathcal U}_{n}+(\tau/2)\big[\varphi_{1}(\mathbf{L}\tau/2)-2\varphi_{2}(\mathbf{L}\tau/2)\big]{{\mathbf{N}}}({\mathcal U}_{n},t_{n})+\tau\varphi_{2}(\mathbf{L}\tau/2){{\mathbf{N}}}(\mu_{2},t_{n}+\tau/2),\\
\mu_{4}=&\;e^{{\mathbf{L}}\tau}{\mathcal U}_{n}+\tau\big[\varphi_{1}(\mathbf{L}\tau)-2\varphi_{2}(\mathbf{L}\tau)\big]{{\mathbf{N}}}({\mathcal U}_{n},t_{n})+2\tau\varphi_{2}(\mathbf{L}\tau){{\mathbf{N}}}(\mu_{3},t_{n}+\tau),
\end{split}
\end{equation*}
with the functions $\varphi_{1},\,\varphi_{2},\,\varphi_{3}$ given as
\begin{equation*}
\varphi_{1}(z)=\frac{e^{z}-1}{z},\quad\varphi_{2}(z)=\frac{e^z-1-z}{z^2},\quad\varphi_{3}(z)=\frac{e^z-1-z-z^2/2}{z^3}.
\end{equation*}

%
\subsection{Linear stability analysis}
The stability analysis of the ETDRK4 method is as follows (see also \cite{krogstad2005generalized,owolabi2016effect,owolabi2016numerical}). For the nonlinear ODE \begin{equation}
{\mathcal U}_{t}=\lambda\,{\mathcal U}+\mathbf{N}({\mathcal U}),
\end{equation}
we suppose that there exists a fixed point ${\mathcal U}_{0}$ such that $\lambda\, {\mathcal U}_{0}+\mathbf{N}({\mathcal U}_{0})=0$. Linearizing about this fixed point leads to
\begin{equation}\label{linearsystem}
{\mathcal U}_{t}=\lambda\,{\mathcal U}+\rho\,{\mathcal U}.
\end{equation}
Here ${\mathcal U}$ is a perturbation to ${\mathcal U}_{0}$, and $\rho=\mathbf{N}'({\mathcal U})$ at ${\mathcal U}={\mathcal U}_{0}$. Then for this ODE, the fixed point ${\mathcal U}_{0}$ is stable if $\text{Re}(\lambda+\rho)<0$.

The application of ETDRK4 method to \eqref{linearsystem} leads to a recurrence relation involving ${\mathcal U}_{n}$ and ${\mathcal U}_{n+1}$. Introducing the notation $x=\rho\tau$, $y=\lambda\tau$, we obtain the following amplification factor
\begin{equation}
\frac{{\mathcal U}_{n+1}}{{\mathcal U}_{n}}=r(x,y)=c_{0}+c_{1}x+c_{2}x^2+c_{3}x^3+c_{4}x^4,
\end{equation}
where
\begin{equation*}
\begin{split}
c_{0}&=e^{y},\\
c_{1}&=4\varphi_{3}(y)-3\varphi_{2}(y)+\varphi_{1}(y)+4e^{y/2}\big(\varphi_{2}(y)-2\varphi_{3}(y)\big)+e^{y}\big(4\varphi_{3}(y)-\varphi_{2}(y)\big),\\
c_{2}&=2\big(\varphi_{1}(y/2)-\varphi_{2}(y/2)+e^{y/2}\varphi_{2}(y/2)\big)\big(\varphi_{2}(y)-2\varphi_{3}(y)\big)\\
&\quad+\big(\varphi_{1}(y)-2\varphi_{2}(y)+2e^{y/2}\varphi_{2}(y)\big)\big(4\varphi_{3}(y)-\varphi_{2}(y)\big),\\
c_{3}&=\Big(\varphi_{2}(y)\big(\varphi_{1}(y/2)-2\varphi_{2}(y/2)\big)+2e^{y/2}\varphi_{2}(y/2)\varphi_{2}(y)\Big)\big(4\varphi_{3}(y)-\varphi_{2}(y)\big)\\
&\quad+\varphi_{1}(y/2)\varphi_{2}(y/2)\big(\varphi_{2}(y)-2\varphi_{3}(y)\big)\\
c_{4}&=\varphi_{1}(y/2)\varphi_{2}(y/2)\varphi_{2}(y)\big(4\varphi_{3}(y)-\varphi_{2}(y)\big).
\end{split}
\end{equation*}
We observe that as $y\to 0$,
\begin{equation*}
\varphi_{1}(y)\to 1,\quad \varphi_{2}(y)\to \frac{1}{2},\quad\varphi_{3}(y)\to \frac{1}{6},
\end{equation*}
then the fourth-order Runge-Kutta method is recovered where
\begin{equation*}
r(x)=1+x+\frac{x^2}{2}+\frac{x^3}{6}+\frac{x^4}{24}.
\end{equation*}

In Fig.\,\ref{stability}, we give the stability region of ETDRK4 method for several different nonpositive $y$, we observe that stability region grows as the absolute value of $y$ increases. We also present in the right figure of Fig.\,\ref{stability} the maximum and minimum eigenvalues of the stiffness matrix $\bf{S}$ for different $N$. It is observed that maximum eigenvalues increase with $N$ approximately as $N^2$ while minimum eigenvalues decrease as $N^{-2}$. Then for the resulted ODE system \eqref{semilinearparabolic} from our spectral-Galerkin method, the stiffness ratio is roughly $N^{4s}$. Thus proper treatment of the linear part is necessary to reduce the excessive restriction on step size and this is the advantage of the ETD method. At the same time, minimum eigenvalues go to $0$ as $N$ increases, then stability region will shrink to the standard fourth-order Runge-Kutta method as $N$ goes to infinity.
\begin{figure}[!h]
\centering
\subfigure[Stability region for different $y$]{
\includegraphics[width=0.45\textwidth]{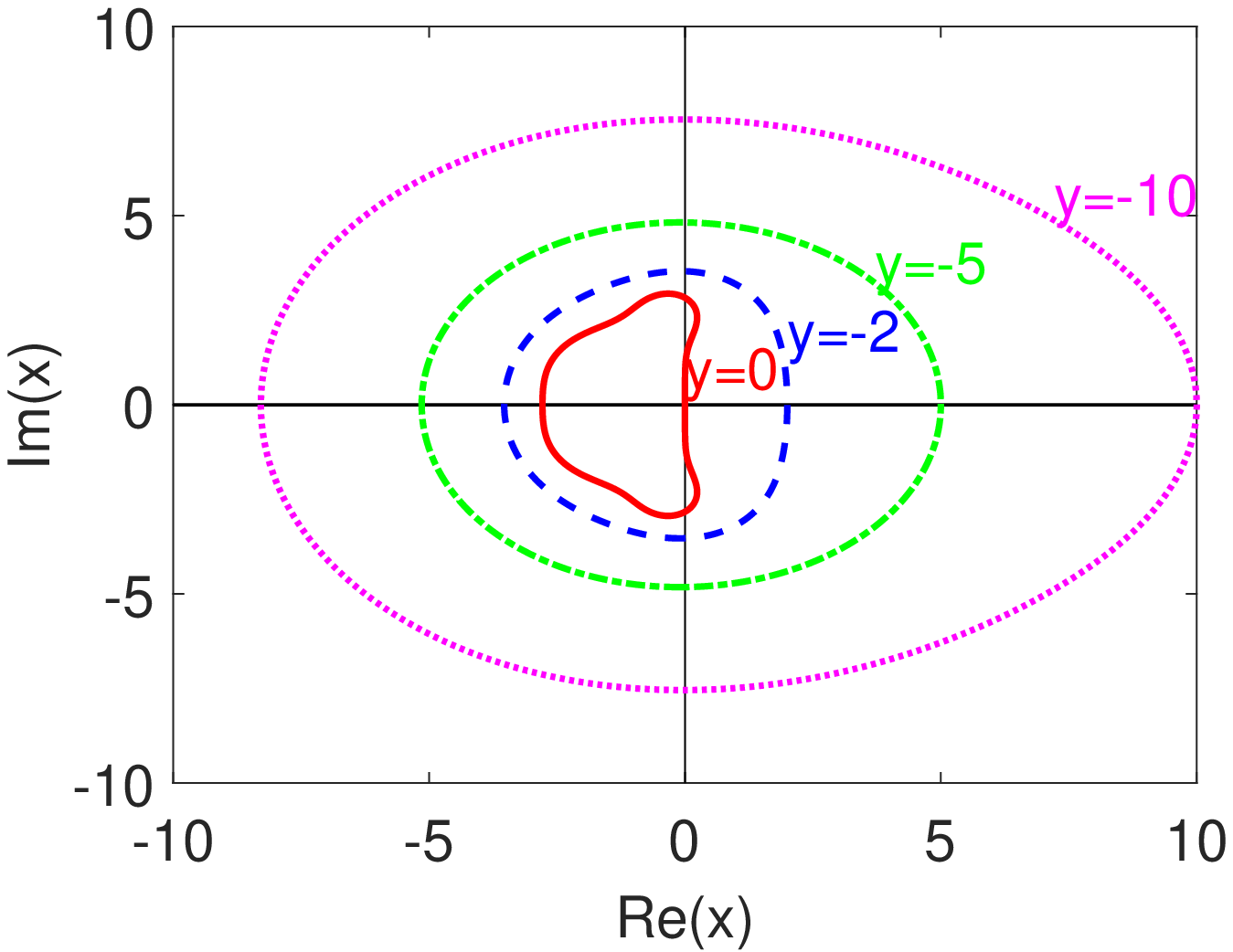}}
\subfigure[Eigenvalues of $\bf{S}$]{
\includegraphics[width=0.45\textwidth]{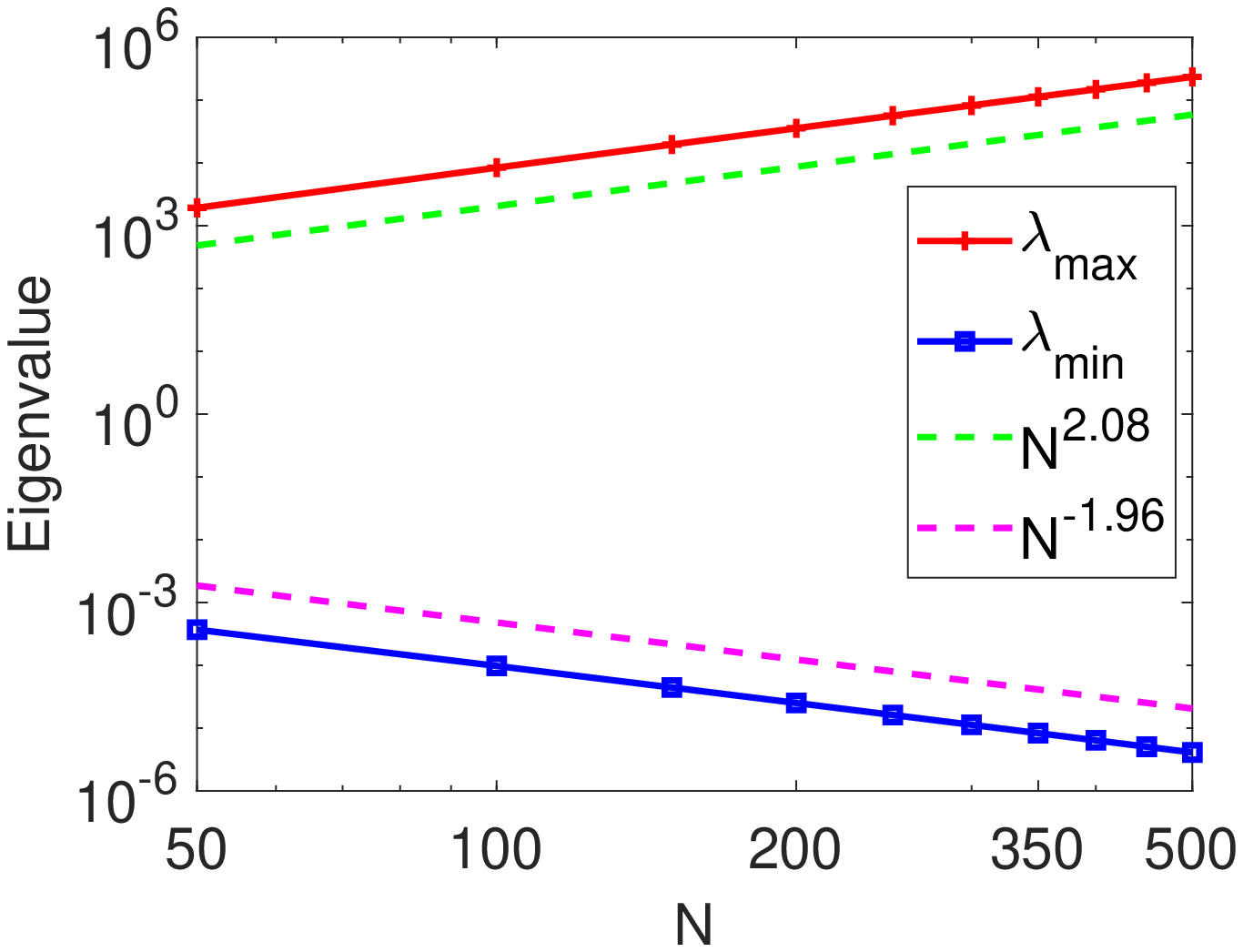}}
\label{stability}
\caption{Linear stability analysis of ETDRK4.}
\end{figure}

%

\section{Numerical examples}

Several numerical examples will be carried out in this section to demonstrate the effectiveness of the spectral-Galerkin methods for reaction-diffusion equations in unbounded domains. Specifically, we study the fractional version of Allen-Cahn equation, Fisher-Kolmogorov equation, Gray-Scott equation and FitzHugh-Nagumo equation in one-, two- and three-dimensions.

\subsection{Convergence test}

We first consider a modified one-dimensional Allen-Cahn equation:
\begin{equation*}
u_{t}+\epsilon^2 (-\Delta)^s u(x)+f(u)=g(x,t),\quad x\in\mathbb{R},\;t\in(0,T].
\end{equation*}
Here ``modified" means we take a different nonlinear term $f(u) = F^\prime (u)$, where $F(u)$ is a double-well potential free energy density. In this paper, we study the double-well potential given by
\begin{equation}\label{Fu}
F(u) = \frac{u^2(u-1)^2}{4}\;\;\; {\rm with}\;\;\; f(u)=\frac{u(u-1)(2u-1)}{2},
\end{equation}
as in \cite{li2017space,liu2018time}. Note that $F(u)$ has energy minima at $u=0 \text{ or }1$, different from the usual double-well potential at $u=\pm 1$, i.e., $F(u)=(u^2-1)^2/4$. This is due to the homogenous boundary condition in \eqref{frde} that $u(x)\to 0$, as $|x|\to\infty$. In order to compute with an exact solution, an extra source term  $g(x,t)$ is chosen
\begin{eqnarray} \label{eq51}
g(x,t)&=& \epsilon^2\exp(-t)\frac{(2\lambda)^{2s}\Gamma(s+\frac{1}{2})}
{\Gamma(\frac{1}{2})}{}_{1}F_{1}\big(s+\frac{1}{2};\frac{1}{2};-\lambda^2x^2\big)
+\exp(-3t)\exp(-3\lambda^2x^2)
\nonumber \\
&& \quad -\frac{3}{2}\exp(-2t)\exp(-2\lambda^2x^2)-\frac{1}{2}\exp(-t)\exp(-\lambda^2x^2),
\end{eqnarray}
where ${}_{1}F_{1}\big(a; b; x)$ is the confluent hypergeometric function defined as
\begin{equation}\label{1f1}
{}_{1}F_{1}(a;b;x)=\sum_{n=0}^\infty\frac{(a)_n\,x^n}{(b)_n\,n!},
\end{equation}
The source term (\ref{eq51}), together with an appropriate initial data, yields an exact solution $u(x,t)=\exp(-t)\exp(-\lambda^2x^2)$.
This example aims to test the temporal convergence rate for the ETDRK4 scheme established for the fractional reaction-diffusion equation.

In this test, we take $\epsilon=0.01$, $T=6$, $\lambda=1$, $N=500$ and several fractional parameters, i.e.,  $s=0.6,\;0.9\text{ and }1.2$. In this example, the relative numerical error is computed via
\begin{equation}
\text{e}_{\infty}=\frac{\lVert u_{N}-u\rVert_{\infty}}{ \lVert u\rVert_{\infty}},\qquad\;\text{e}_{L^2}=\frac{\lVert u_{N}-u\rVert_{2}}{\lVert u\rVert_{2}}.
\end{equation}
Numerical results with respect to time step size $\tau$ is given in Fig.\,\ref{convergence2}, where the convergence order is shown to be approximately $\tau^{4}$, which is in good agreement with the theoretical prediction.

We also present in Fig.\,\ref{convergence2} numerical errors in two-dimensional space with same parameters, where the exact solution is chosen as
\begin{equation*}
u(x,y,t)=\exp(-t)\exp(-(x^2+y^2)).
\end{equation*}
Similar results are obtained.

\begin{figure}[!h]
\centering
\subfigure[$s=0.6$]{
\includegraphics[width=0.31\textwidth]{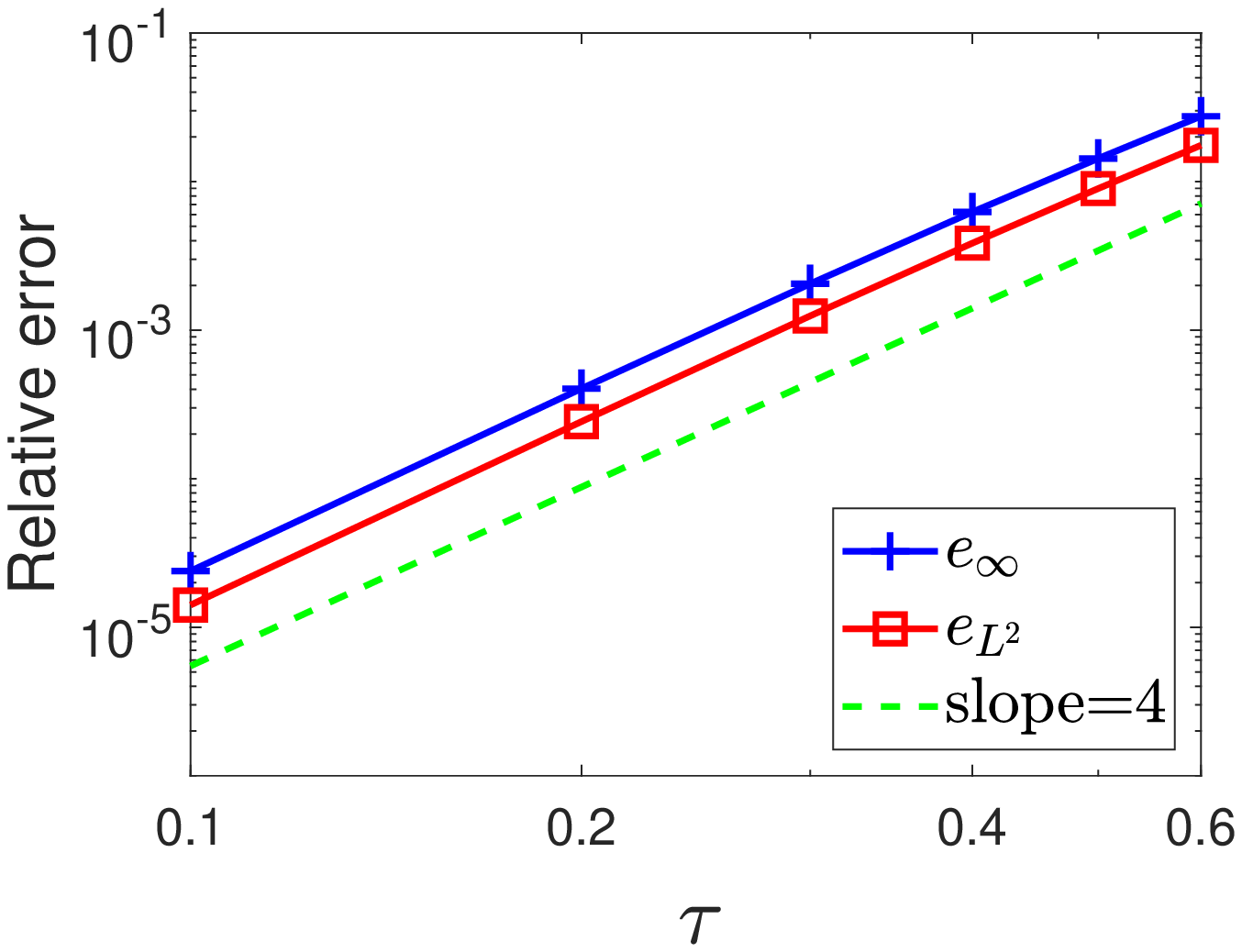}}
\subfigure[$s=0.9$]{
\includegraphics[width=0.31\textwidth]{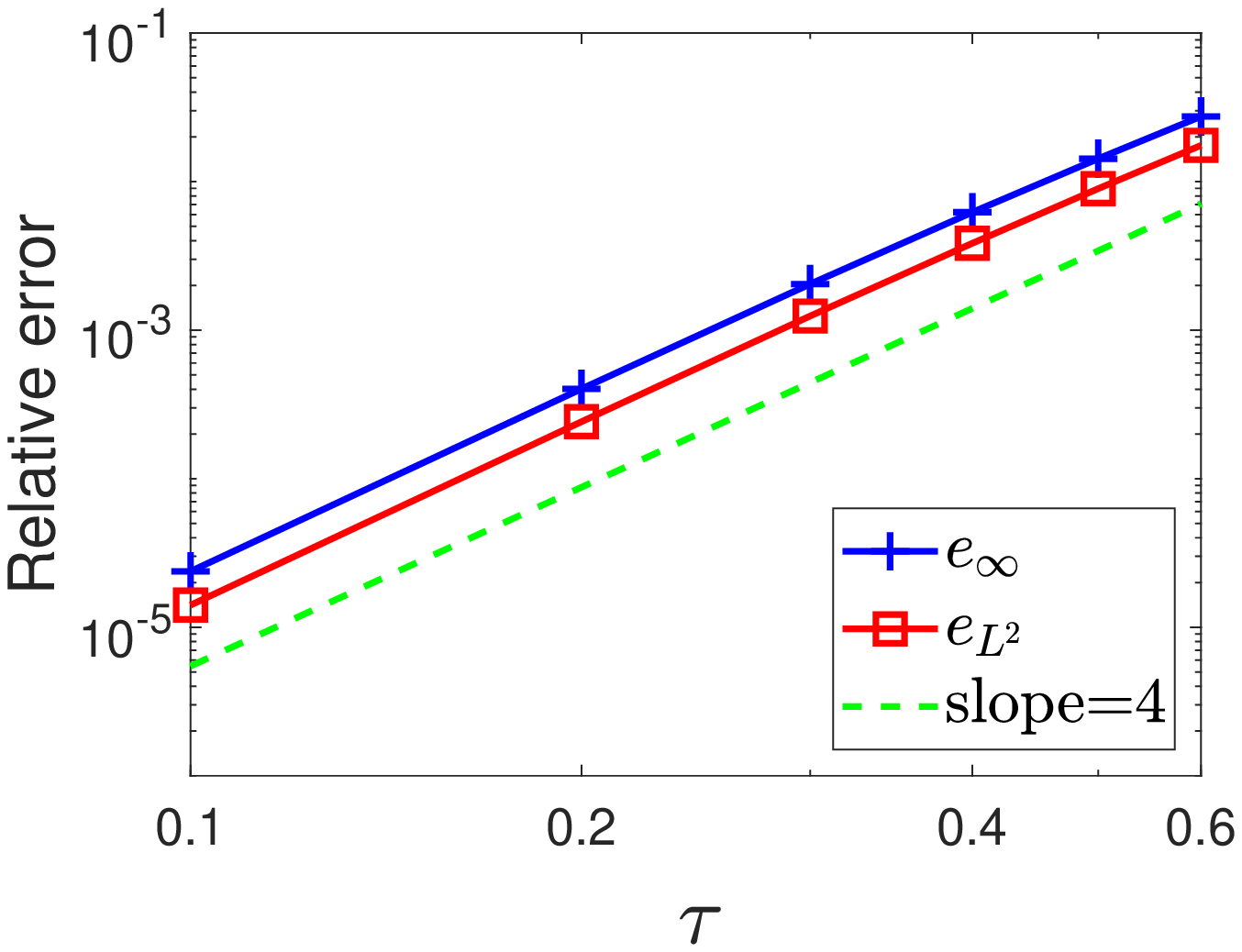}}
\subfigure[$s=1.2$]{
\includegraphics[width=0.31\textwidth]{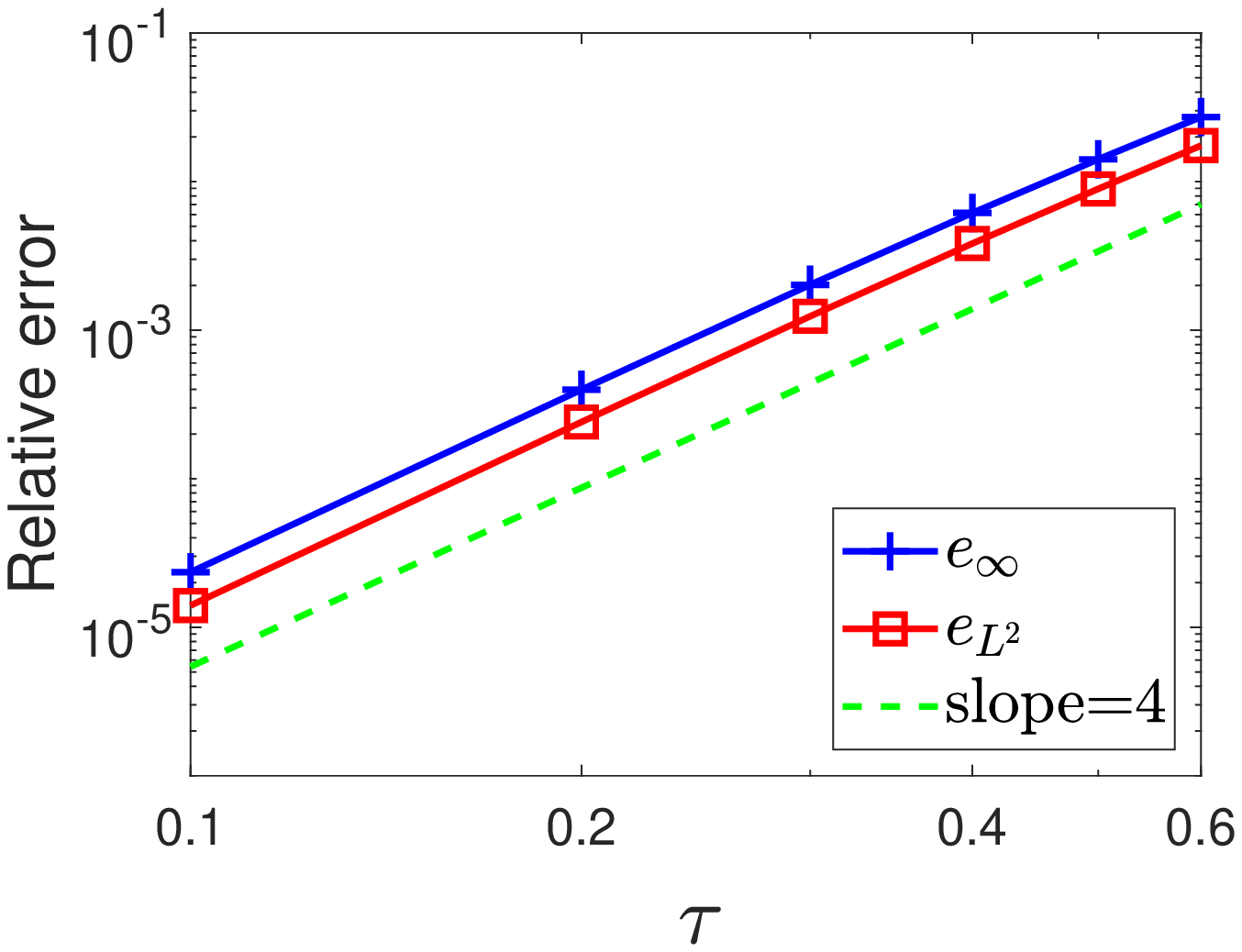}}
\label{convergence1}
\caption{Convergence of numerical errors for 1-d problem.}
\end{figure}

\begin{figure}[!h]
\centering
\subfigure[$s=0.6$]{
\includegraphics[width=0.31\textwidth]{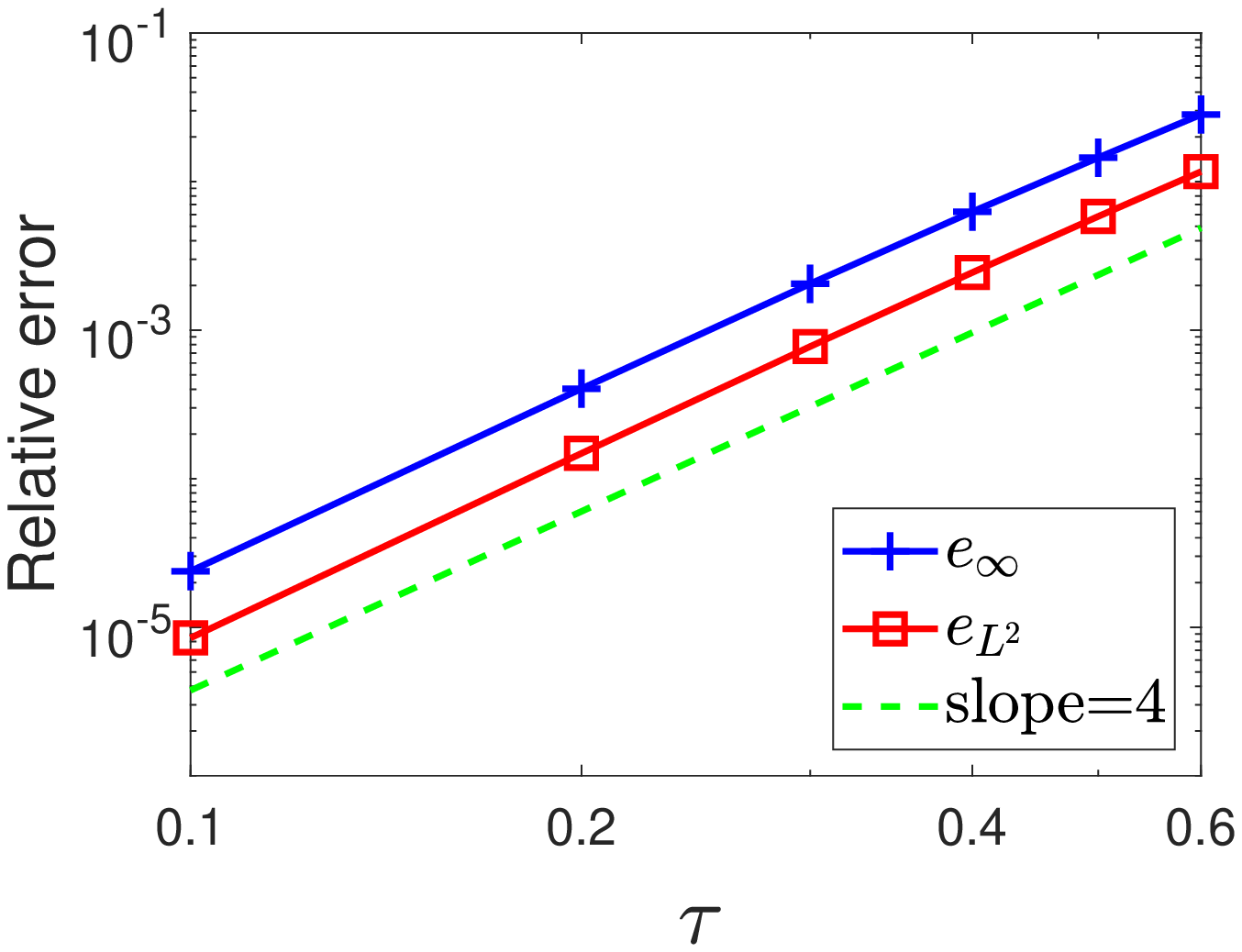}}
\subfigure[$s=0.9$]{
\includegraphics[width=0.31\textwidth]{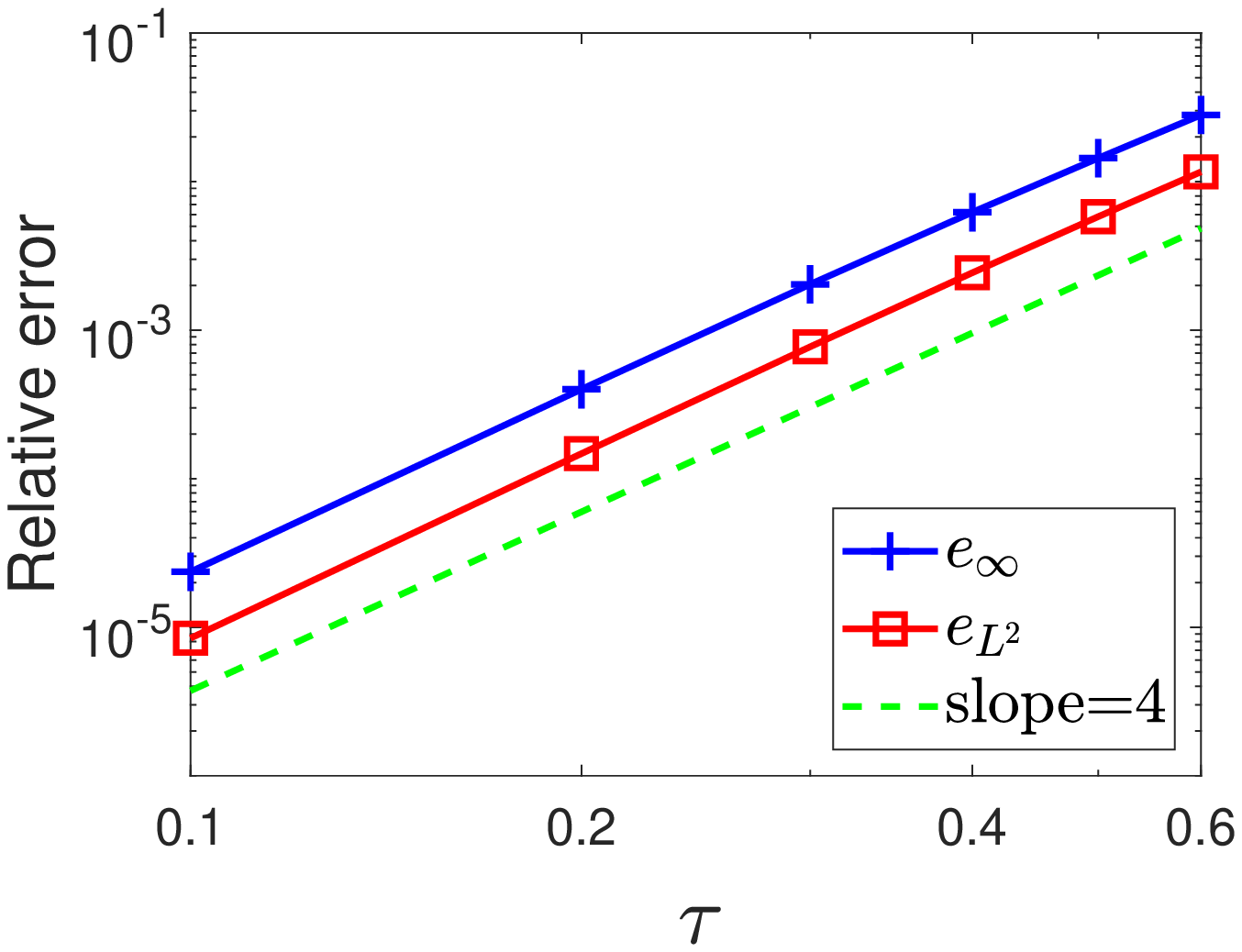}}
\subfigure[$s=1.2$]{
\includegraphics[width=0.31\textwidth]{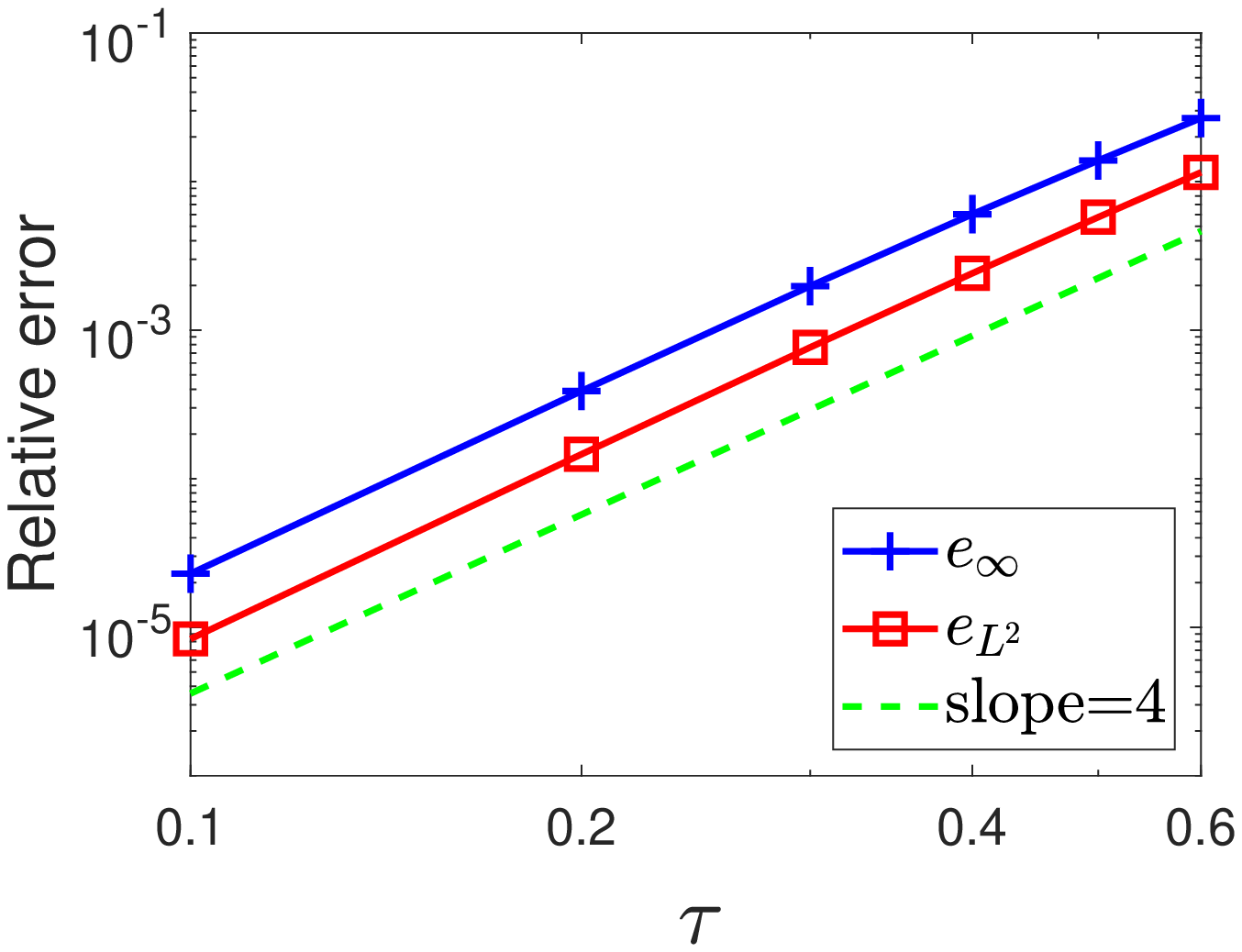}}
\label{convergence2}
\caption{Convergence of numerical errors for 2-d problems.}
\end{figure}

\subsection{Fisher-Kolmogorov equation}
In this part, we apply the spectral-Galerkin method to the Fisher-Kolmogorov equation, i.e., \eqref{frde}  with a quadratic nonlinear reaction term:
\begin{equation} \label{eq53}
f(u)=ru\left(1-\frac{u}{K}\right),
\end{equation}
where $ r$ is the intrinsic growth rate of a species and $K$ is the environment carrying capacity, representing the maximum sustainable population density.
The Fisher-Kolmogorov equation has attracted a lot of interest due to its wide application to model the growth and spreading of biological species. We first study this equation in one-dimension, and take $D=0.1$, $r=0.25$, $K=1$ and $N=1000$. Numerical results are given in Fig.\,\ref{FK1D}. In $(a)$, the profiles of the numerical solution for $s=0.8$ are given for various times, e.g., $t=0,\,10,\,20$ and $30$, where an accelerating front can be observed as time evolves. We then show in $(b)$ the position of $x$, denoted as $x_{0}$, where $u(x)$ first exceeds $0.5$ versus $t$ on a semi-log plot. It is observed that the front expands exponentially with time, in agree with the results in \cite{del2003front} for one-sided fractional derivatives. Numerical results of $u(x)$ at final time $T=30$ is given in $(c)$ for several different $s$, we see a larger accelerating front for smaller $s$, in agreement with $(b)$. We finally show in $(d)$ the asymptotic behavior of the numerical solution when $|x|$ is relatively large. Different from the classical case when $s=1$, an algebraical decay is observed for fractional $s$, which is approximately $|x|^{-2s-1}$. This can then be explained by the presence of non-Gaussian diffusion when $s$ is not an integer.
\begin{figure}[!h]
\centering
\subfigure[Profile of solution for $s=0.8$]{
\includegraphics[width=0.45\textwidth]{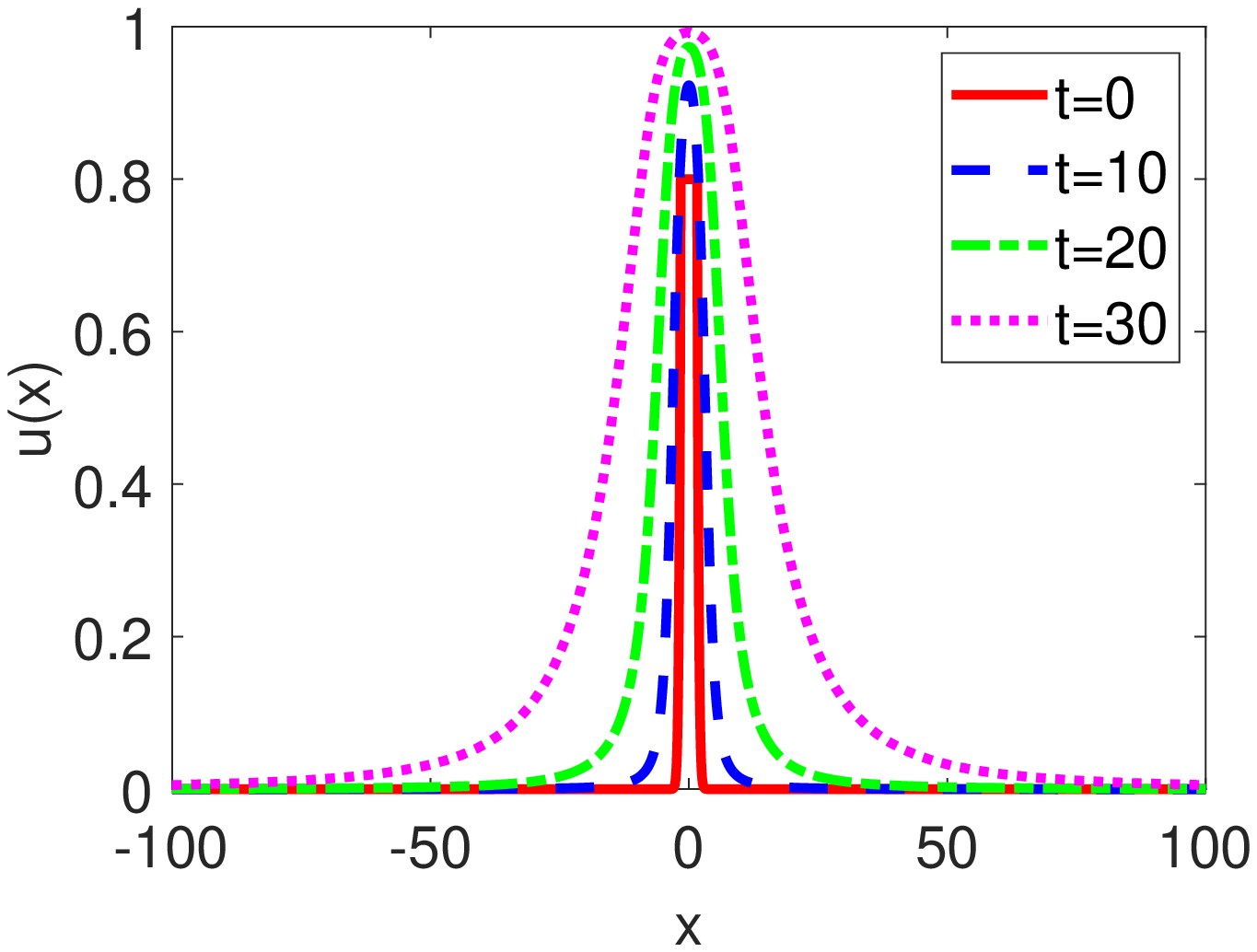}}
\subfigure[Accelerating front for $u(x)=0.5$]{
\includegraphics[width=0.45\textwidth]{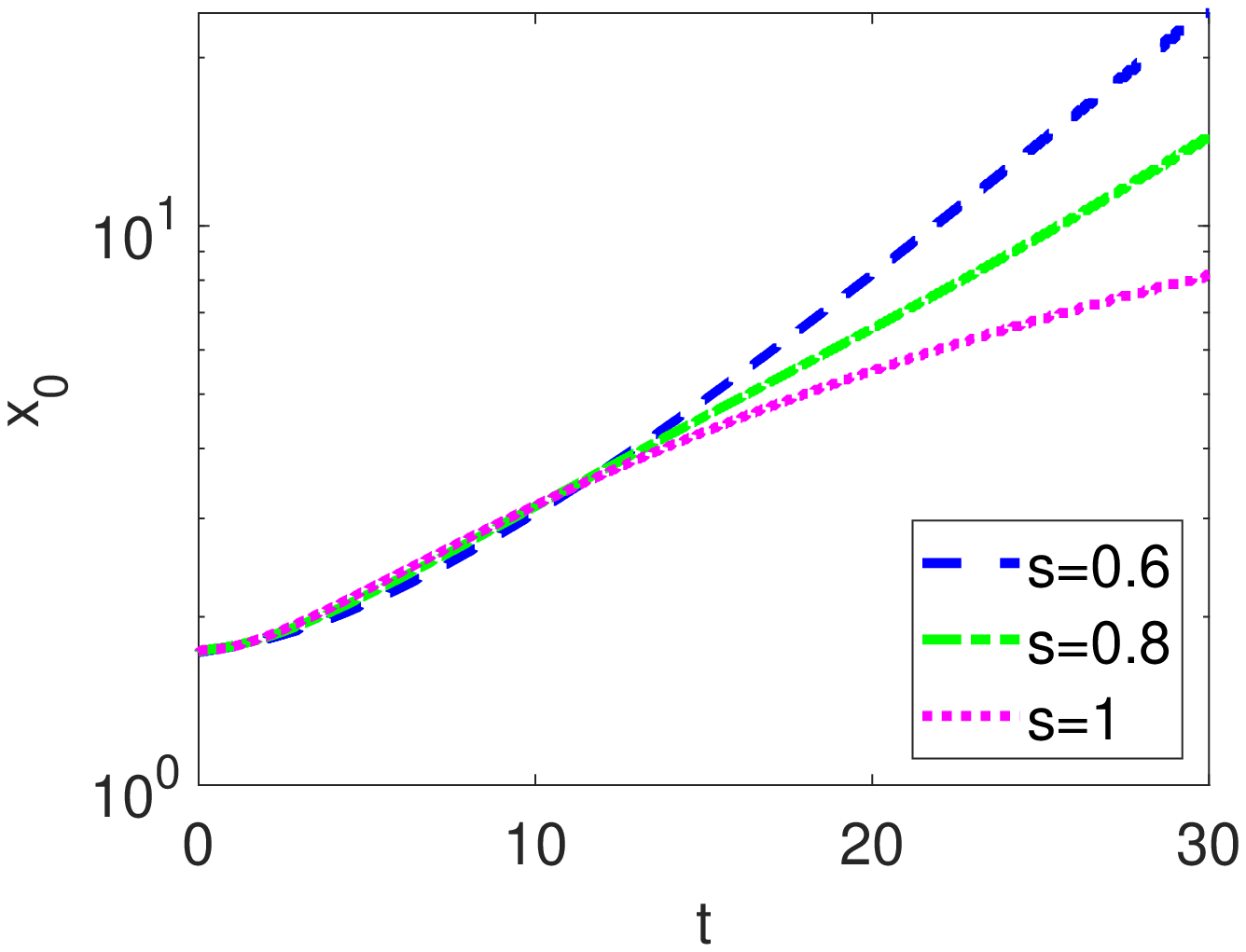}}
\subfigure[Final state when $T=30$.]{
\includegraphics[width=0.45\textwidth]{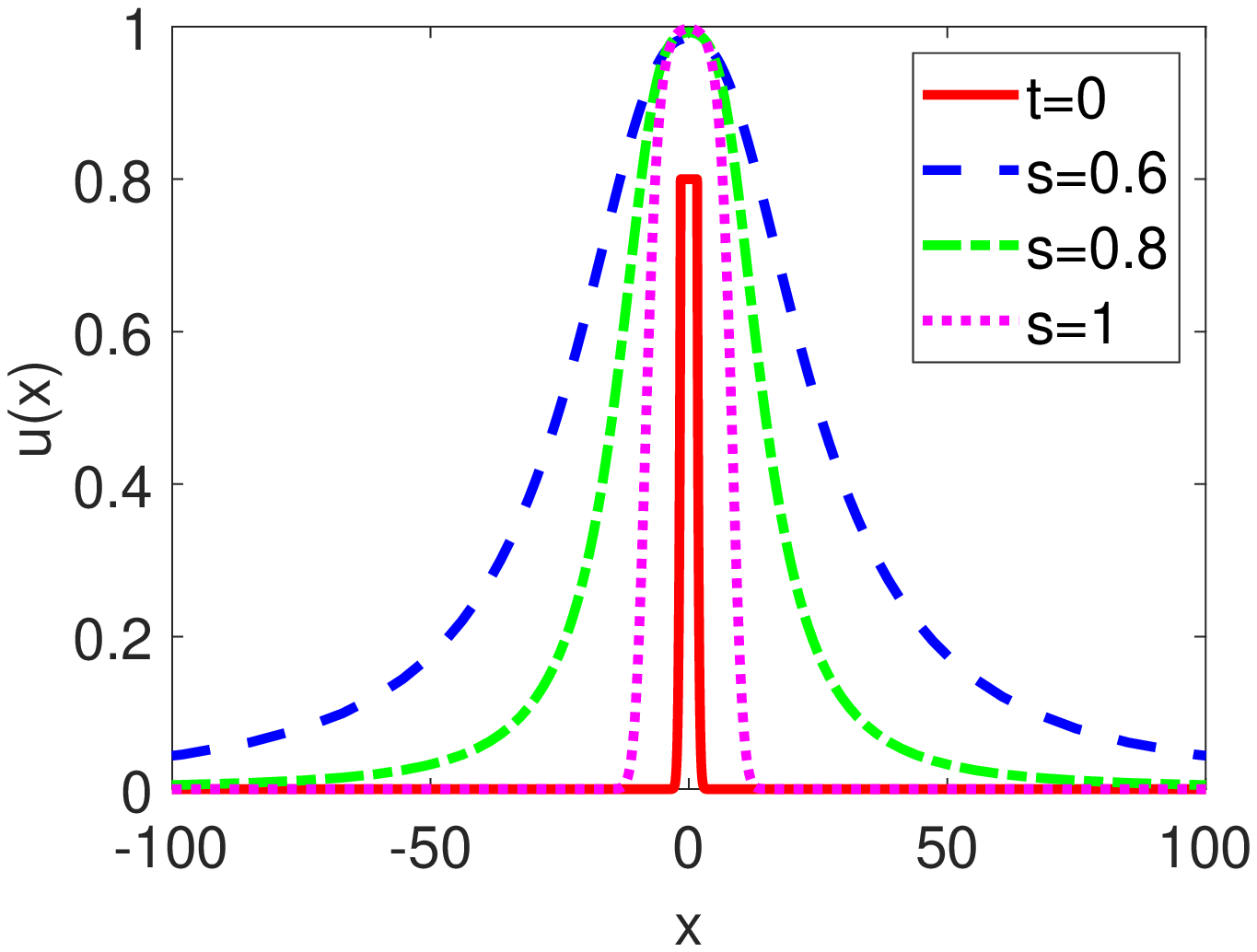}}
\subfigure[Asymptotic behavior for $|x|$ large]{
\includegraphics[width=0.45\textwidth]{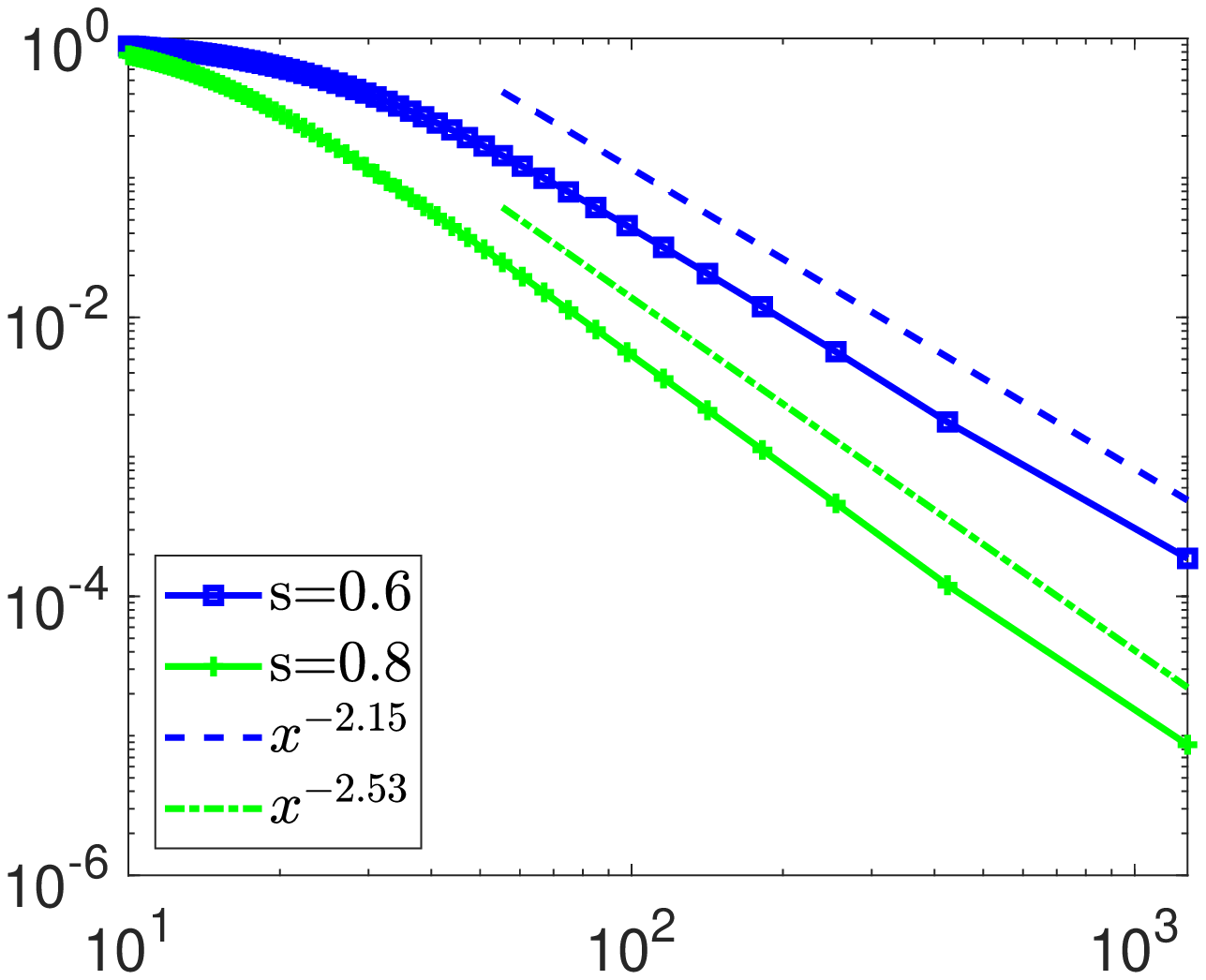}}
\caption{Numerical results for Fisher-Kolmogorov equation in one-dimension.}
\label{FK1D}
\end{figure}

Numerical results for Fisher-Kolmogorov equation in two-dimensional space are given in Fig.\,\ref{FK2D}, where we give the contours of the numerical solution at final time $T=30$. Here the initial condition is chosen as
\[
u_{0}=\min\Big\{0.8,10\exp(-x^2-y^2)\Big\}.
\]
The other parameters are the same as in one-dimension. In the left figure of Fig.\,\ref{FK2D}, we present the numerical result for the fractional Laplace operator $(-\Delta)^{s}$, where $s=0.8$. We observe that the contours of the solution are circles since fractional Laplacian is a symmetric operator. We also present in the right figure of Fig.\,\ref{FK2D} the numerical solution for Riesz fractional derivatives in two directions, i.e., we study the following equation
\begin{equation}
\frac{\partial u}{\partial t}=D_{1}\frac{\partial^{2\alpha} u}{\partial|x|^{2\alpha}}+D_{2}\frac{\partial^{2\beta} u}{\partial|y|^{2\beta}}+f(u).
\end{equation}
In order to compare with the fractional Laplace operator, here we choose $D_1=D_2=D=0.1$ and $\alpha=0.75,\,\beta=0.85$. Asymmetry exhibits in the right figure of Fig.\,\ref{FK2D}, although symmetry can be observed along each axis. We also observe that the accelerating front spreads faster for a smaller fractional derivative order, in agreement with our results for the one-dimensional examples.
\begin{figure}[!h]
\centering
\subfigure[Fractional laplacian $s=0.8$]{
\includegraphics[width=0.45\textwidth]{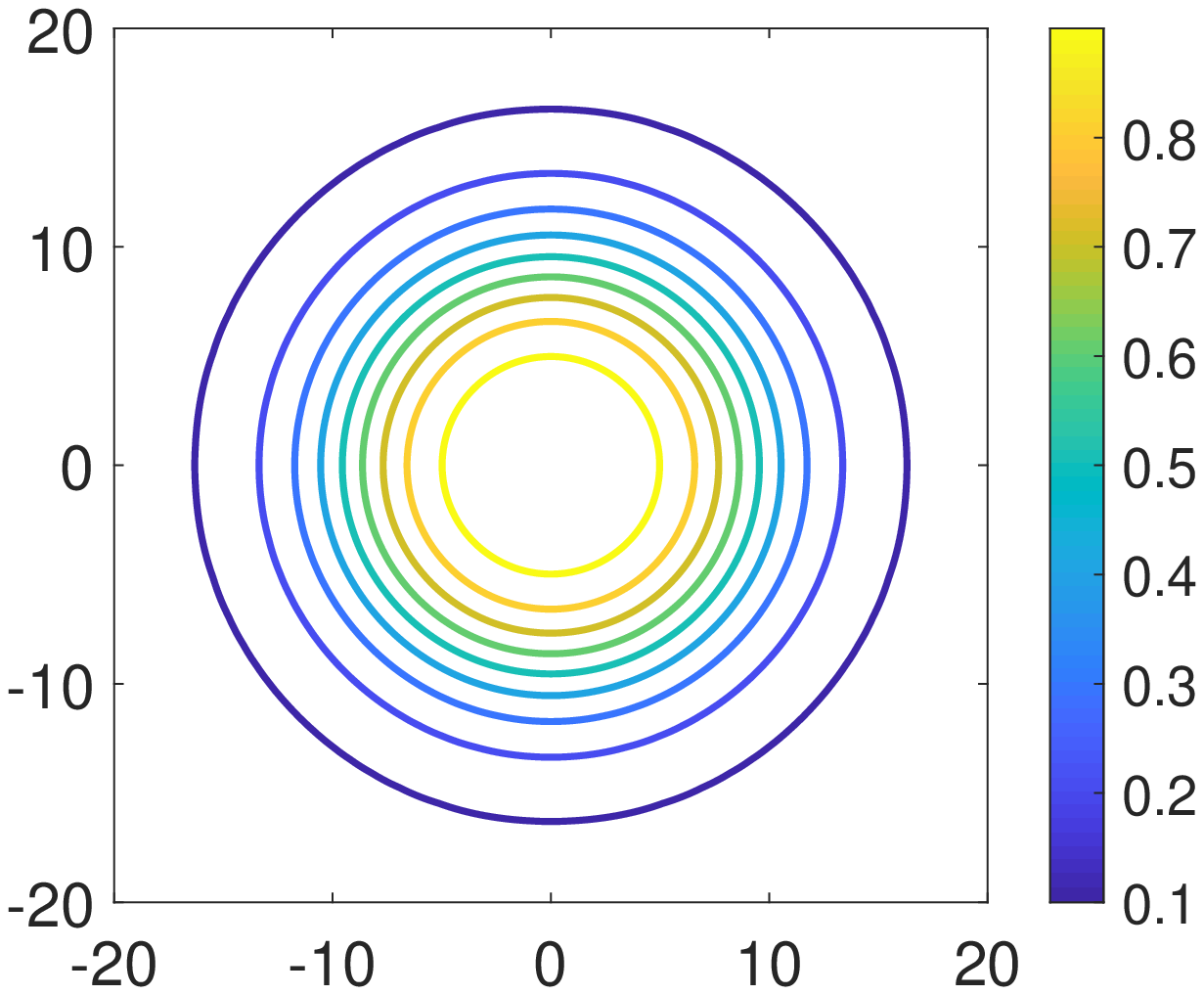}}
\subfigure[Riesz derivative $\alpha=0.75,\,\beta=0.85$]{
\includegraphics[width=0.45\textwidth]{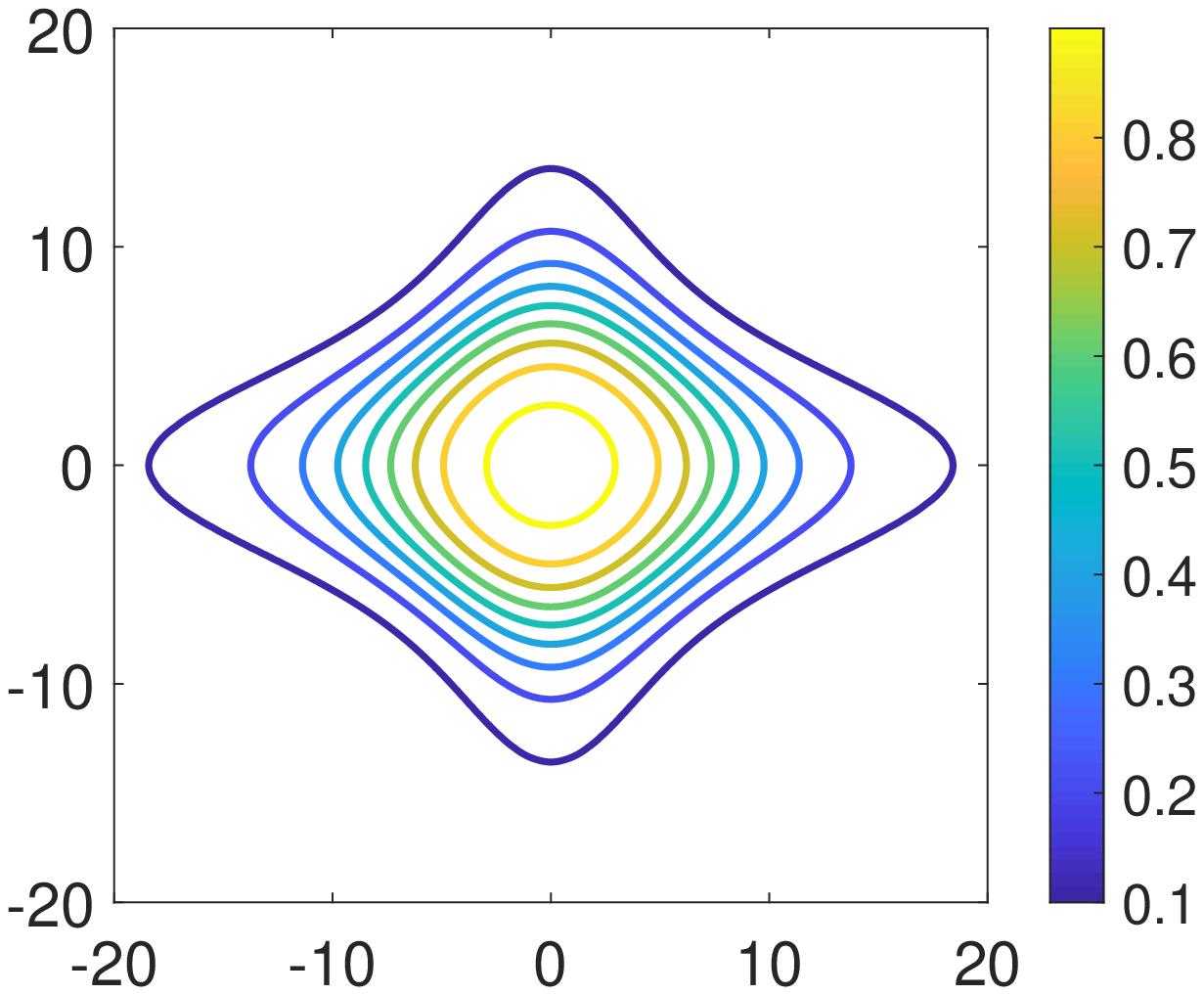}}
\caption{Numerical results for Fisher-Kolmogorov equation in two-dimension.}
\label{FK2D}
\end{figure}

\subsection{Allen-Cahn equation}\label{2dalen}
We consider in this example the spectral-Galerkin method for the fractional-in-space Allen-Cahn equation which is the classical phase field model. Here, the reaction term is given in \eqref{Fu}.

In our first numerical experiment, we set $N_{x}=N_{y}=500$, $\epsilon=0.02$, $\tau=0.1.$ With the random initial data
\begin{equation}
u_{0}=0.5+0.1(\text{rand}-0.5),
\end{equation}
 we show the numerical results in Fig.\,\ref{allen2drand} for different $s$. Although the solution domain is infinity, we display in Fig.\,\ref{allen2drand} the numerical results in $[-1,1]^2$. In this experiment, we are primarily concerned with the effect of $s$ on the speed of evolution. It is observed  from Fig.\,\ref{allen2drand} that the fractional Allen-Cahn models with larger values of $s$ produce a rapid movement to larger bulk regions. We also note that reducing the fractional derivative order $s$ leads to a thinner interfacial layer.

\begin{figure}[!h]
{\hspace*{-1.5cm}\includegraphics[width=0.9\textwidth]{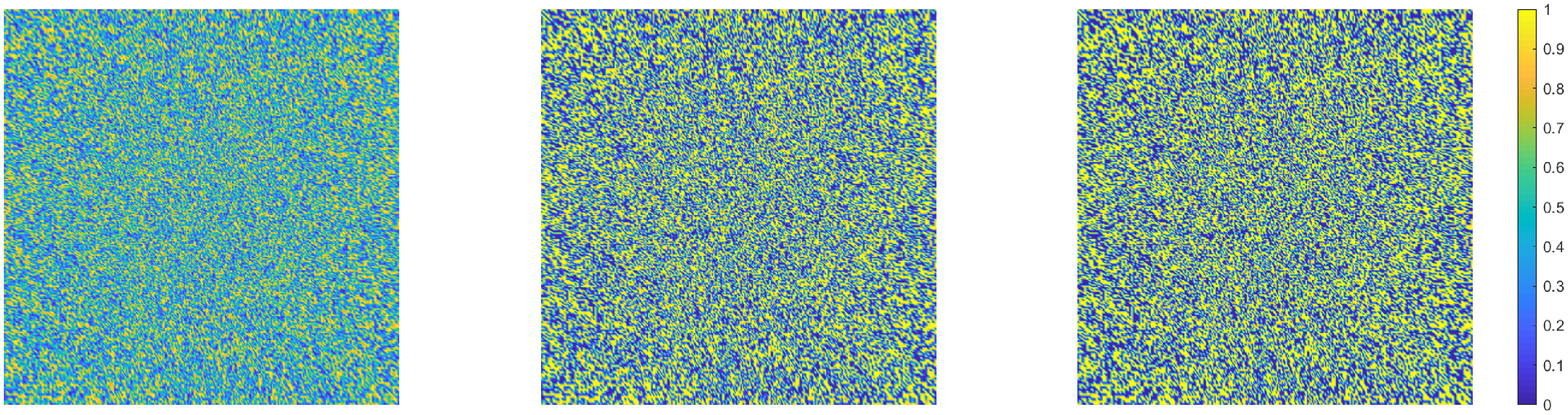}}
{\hspace*{-1.5cm}\includegraphics[width=0.9\textwidth]{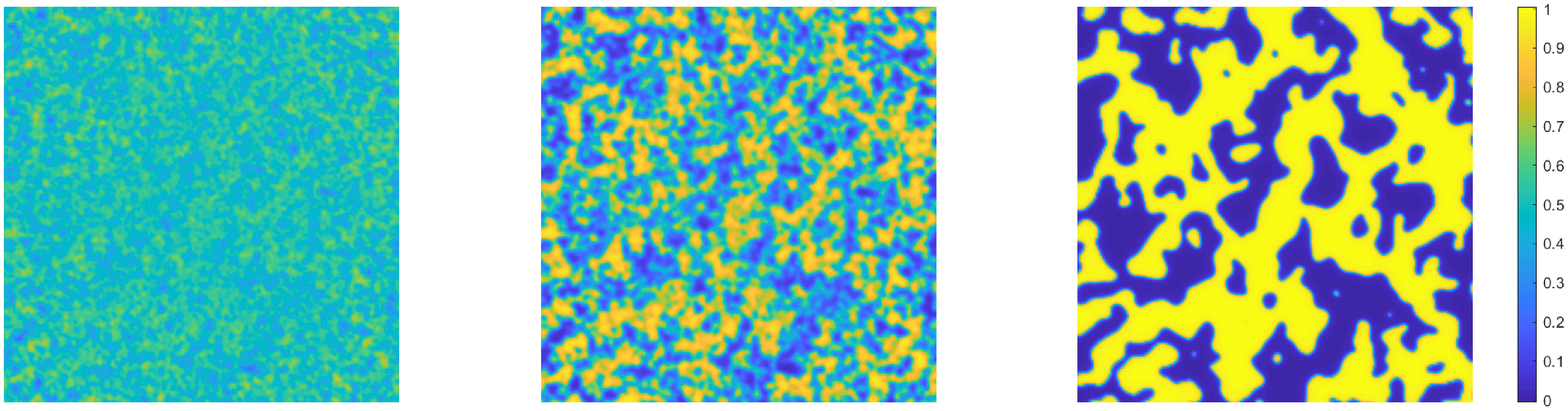}}
{\hspace*{-1.5cm}\includegraphics[width=0.9\textwidth]{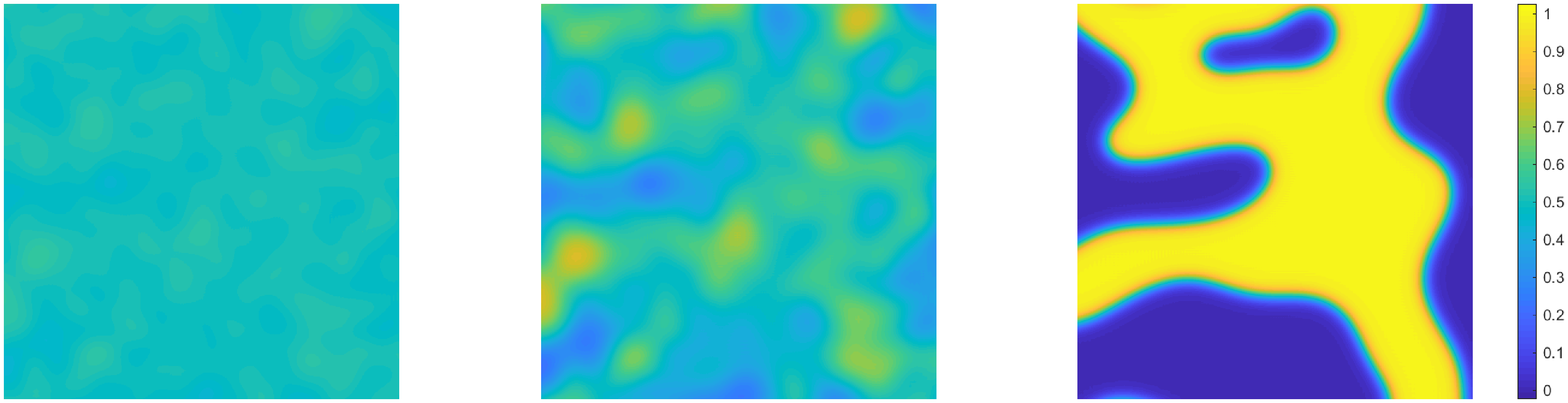}}
\caption
{\small  The snapshots of the numerical solution for space fractional Allen-Cahn equation. Top row: $s=0.3$, middle: $s=0.6$, bottom: $s=0.9$. From left to right: $t=10$, $20$, $100$.}\label{allen2drand}
\end{figure}

In next test for fractional Allen-Cahn equation, we study the evolution of mean curvature for
\begin{equation}\label{allen2dcurvature}
u_{t}=-(-\Delta)^s u(x)+\frac{f(u)}{\epsilon^2}.
\end{equation}
Here the initial condition is taken as
\begin{equation}\label{allen2dcurvatureinitial}
u_{0}=\frac{1}{2}\bigg(1-\tanh\Big(\frac{\sqrt{x^2+y^2}-0.6}{\sqrt{2}\epsilon}\Big)\bigg).
\end{equation}
We set $\epsilon=0.01$, $\tau=10^{-4}$ and $N_{x}=N_{y}=500$, numerical results are given in Fig.\,\ref{allen2dradius}. As time evolves, the radius when $u\geq 0.5$ is recorded for $s=0.8$. Results in Fig.\ref{allen2dradius} $(a)$ shows the evolution of the numerical solution with the initial value given in \eqref{allen2dcurvatureinitial} at $t=0,\,0.2,\,0.5,\,0.8$. The 0.5 level contour lines of $u(x,y,t)$ are plotted in $(b)$. Here contour of circles are observed since we use the fractional Laplace operator. At the same time, the circle shrinks with time. In $(c)$, quantitative change of radius$^2$, i.e., area, with respect to time for different $s$ are given. We know for $s=1$, the evolution of the mean curvature flow is radius\,=$\sqrt{R_{0}^2-2t}$. In our case, here $R_{0}=0.6$. Thus, singularity for $s=1$ happens at $t=0.18$, which is the disappearing time, as is also shown in our numerical result in figure $(c)$. In fact, linear decrease between area and time is observed in all cases, while the slope is different. When decreasing $s$, a smaller slope is obtained.
\begin{figure}[!th]
\centering
{\subfigure[Evolution of $u(x,y)$ for $s=0.8$ at $t=0,\,0.2,\,0.5$ and $0.8$]{
{\hspace*{-1cm}}\includegraphics[width=1.1\textwidth]{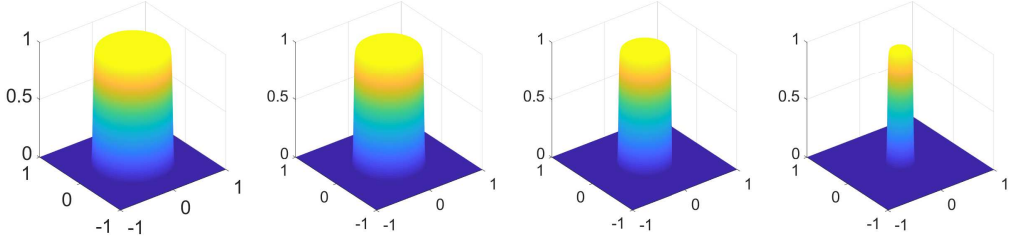}}}
{\subfigure[Contour of level $0.5$ for $s=0.8$ at different times]{
\includegraphics[width=0.48\textwidth]{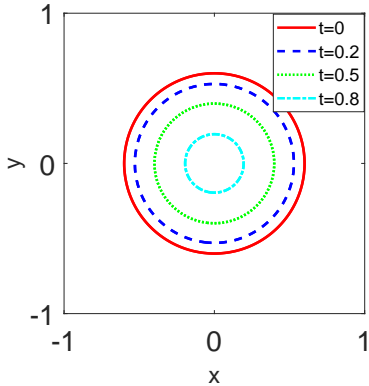}}
\subfigure[Radius$^{2}$ change with time for different $s$]{
\includegraphics[width=0.48\textwidth]{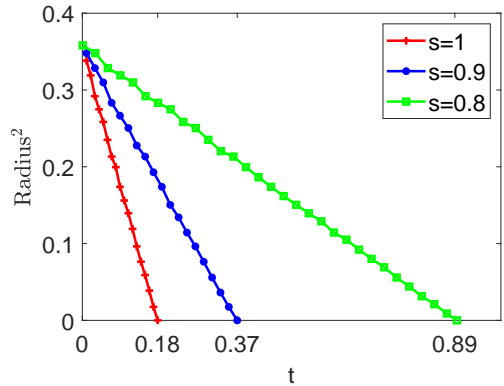}}}
\label{allen2dradius}
\caption{Numerical result of \eqref{allen2dcurvature} with a smooth initial value.}
\end{figure}

\subsection{Gray-Scott equation}
The space-fractional Gray-Scott equation is given by
\begin{equation}\label{gs}
\left\{\begin{array}{ll}
  \frac{\partial U}{\partial t}=-K_{u}(-\Delta)^{\alpha/2}U(x)-UV^2+F(1-U),\vspace{1.2ex}\\
  \frac{\partial V}{\partial t}=-K_{v}(-\Delta)^{\beta/2}V(x)+UV^2-(F+\kappa)V,
  \end{array}\right.
\end{equation}
where $K_{u},\,K_{v}$ are the two diffusion coefficients and $F,\,\kappa$ are the dimensionless positive constants. This model corresponds to the following two reactions
\begin{equation}
\begin{split}
U+2V&\to 3V,\\
V&\to P,
\end{split}
\end{equation}
where $U,\,V$ and $P$ represent different chemical species. During the two reactions, $U$ is continuously supplied at a given feed rate $F$ and $V$ is removed at a given kill rate $F+\kappa$. It can be easily obtained that $(U^{*},V^{*})=(1,0)$ is a trivial stable point for this problem. In order to satisfy the homogeneous boundary conditions as $|x|\to\infty$, we need to modify the reaction terms in \eqref{gs} a little bit. By letting $u=1-U$, $v=V$, we have for the first equation in \eqref{gs} that
\begin{equation}
  -\frac{\partial u}{\partial t}=K_{u}(-\Delta)^{\alpha/2}u(x)-(1-u)v^2+Fu,
\end{equation}
i.e.,
\begin{equation}
  \frac{\partial u}{\partial t}=-K_{u}(-\Delta)^{\alpha/2}u(x)+(1-u)v^2-Fu,
\end{equation}
Consequently, the Gray-Scott equation \eqref{gs} is turned to
\begin{equation}\label{gs1}
\left\{\begin{array}{ll}
  \frac{\partial u}{\partial t}=-K_{u}(-\Delta)^{\alpha/2}u(x)+(1-u)v^2-Fu,\vspace{1.2ex}\\
  \frac{\partial v}{\partial t}=-K_{v}(-\Delta)^{\beta/2}v(x)+(1-u)v^2-(F+\kappa)v.
  \end{array}\right.
\end{equation}
Now $(u^{*},v^{*})=(0,0)$ is a stable point of \eqref{gs1}. In our example, we choose the following initial condition
\begin{equation}
(u_{0},v_{0})=
\left\{\begin{array}{ll}
  (0.5,0.25),&\quad x^2+y^2<1\\
  (0,0),&\quad \text{else},
  \end{array}\right.
\end{equation}
as a permutation of the trivial stable state. Other parameters are $K_{u}=2\times 10^{-5}$, $K_{v}=K_{u}/2$, $F=0.03$, $N=800$, $\tau=0.1$, while we vary $s$ and $\kappa$. Here the ratio between diffusion coefficients $K_u/K_v>1$ is chosen since the model is known to generate different mechanisms of pattern formation depending on the values of feed rate $F$ and kill rate $\kappa$. The evolutions of numerical approximations for $u$ with different $s$ and $\kappa$ are given in Fig. \ref{gsnum55} and \ref{gsnum61}. The domain of interest is taken as $[-3,3]^2$ in all figures in order to display the propagation with time. Fig.\,\ref{gsnum55} shows that the permutations propagate outward for all values of $s$ and a reduction in the fractional order leads to a decrease in the velocity of the propagation. By comparison with Fig.\,\ref{gsnum61} for same $s$, we observe that the values of $u$ concentrate on a smaller value with increasing kill rate $\kappa$, and when $\kappa=0.061$, numerical approximations of $s=1$ even go to the homogenous steady sate.
\begin{figure}[!htbp]
\subfigure{
\centering
{\hspace*{-1.2cm}}{\includegraphics[width=1\textwidth]{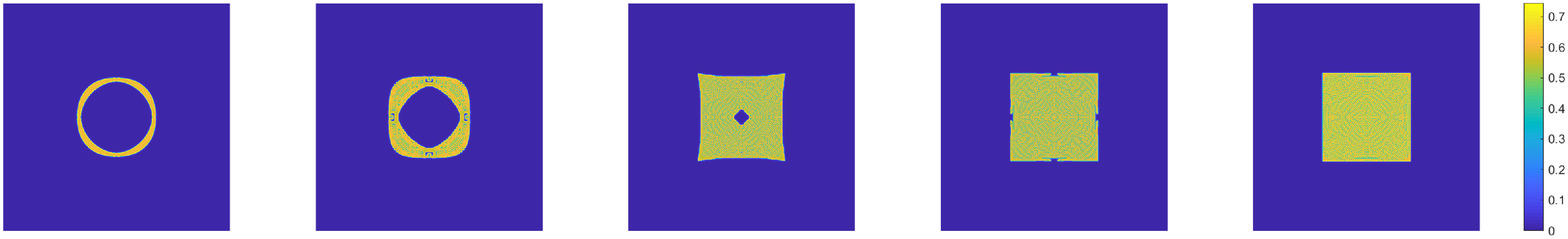}}}
\vskip -5pt
\subfigure{
\centering
{\hspace*{-1.2cm}}{\includegraphics[width=1\textwidth]{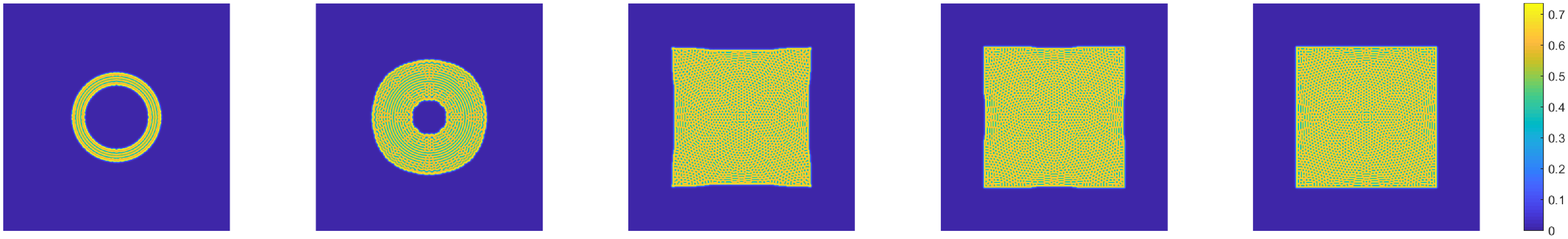}}}
\vskip -5pt
\subfigure{
\centering
{\hspace*{-1.2cm}}{\includegraphics[width=1\textwidth]{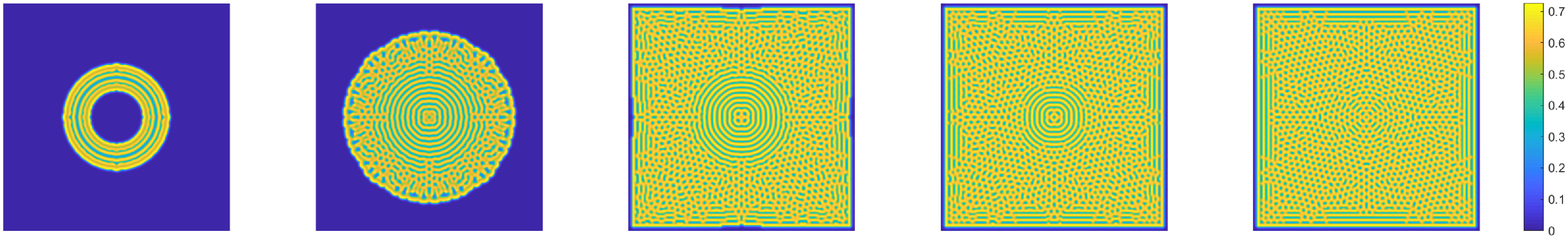}}}
\vskip -5pt
\caption{Numerical approximations of Gray-Scott equation with $\kappa=0.055$ with different values of $s$. Top: $s=0.75$, middle: $s=0.85$, bottom: $s=1$. Time from left to right: $t=1000,3000,9000,15000,30000$.}
\label{gsnum55}
\end{figure}

\begin{figure}[!htbp]
\subfigure{
\centering
{\hspace*{-1.2cm}}{\includegraphics[width=1\textwidth]{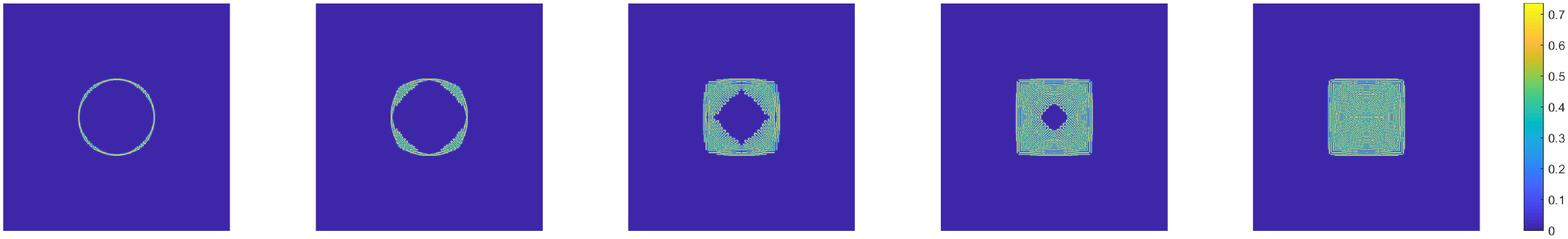}}}
\vskip -5pt
\subfigure{
\centering
{\hspace*{-1.2cm}}{\includegraphics[width=1\textwidth]{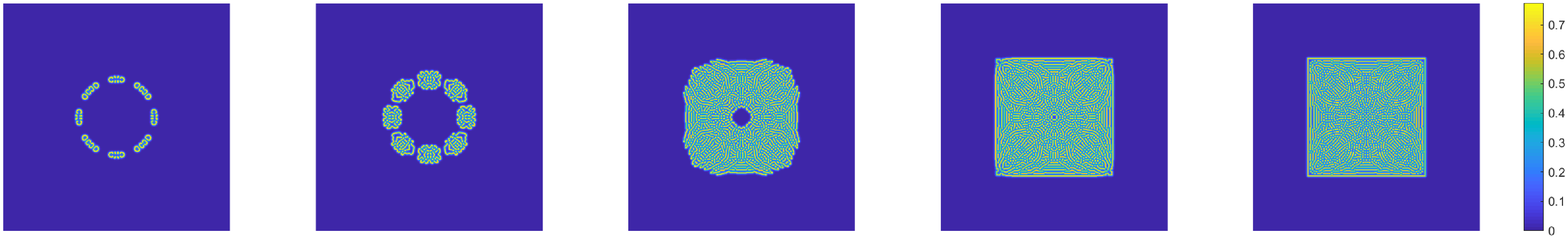}}}
\vskip -5pt
\subfigure{
\centering
{\hspace*{-1.2cm}}{\includegraphics[width=1\textwidth]{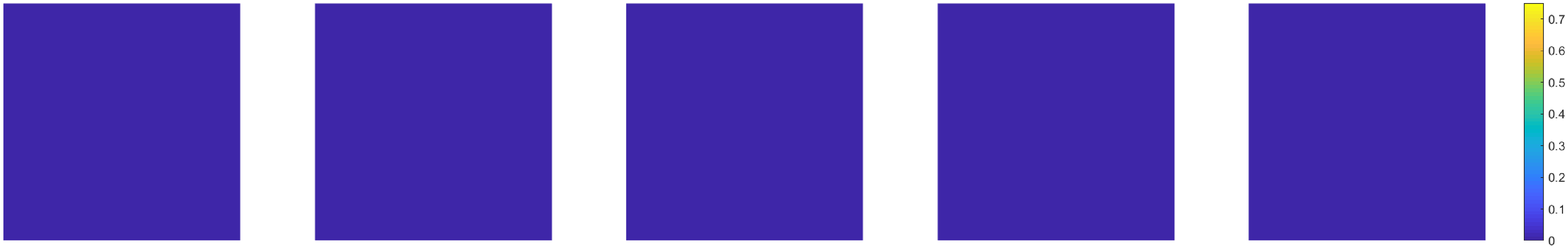}}}
\vskip -5pt
\caption{Numerical approximations of Gray-Scott equation with $\kappa=0.061$ with different values of $s$. Top: $s=0.75$, middle: $s=0.85$, bottom: $s=1$. Time from left to right: $t=1000,3000,9000,15000,30000$. We remark that the third-row figure with $s=1$ represents the homogeneous steady state and is displayed with the same colorbar as the first- and second-row figures for comparison.}
\label{gsnum61}
\end{figure}

For easy comparison of the pattern formation at final time $T=30000$, we present in Fig.\ref{gspattern} with the domain chosen as $[0,1]^2$. We observe from Fig.\,\ref{gspattern} that different patterns are developed for different combinations of $\kappa$ and $s$. For $\kappa=0.55,\,s=0.85$, almost only spot patterns are observed, while a mixed pattern of stripes and spots appears for $\kappa=0.61,\,s=0.85$. And in the case of $\kappa=0.61,\,s=0.75$, more thin and long stripes emerge.

\begin{figure}[!h]
\centering
\subfigure[$\kappa=0.55,\,s=0.85$]{
\includegraphics[width=0.31\textwidth]{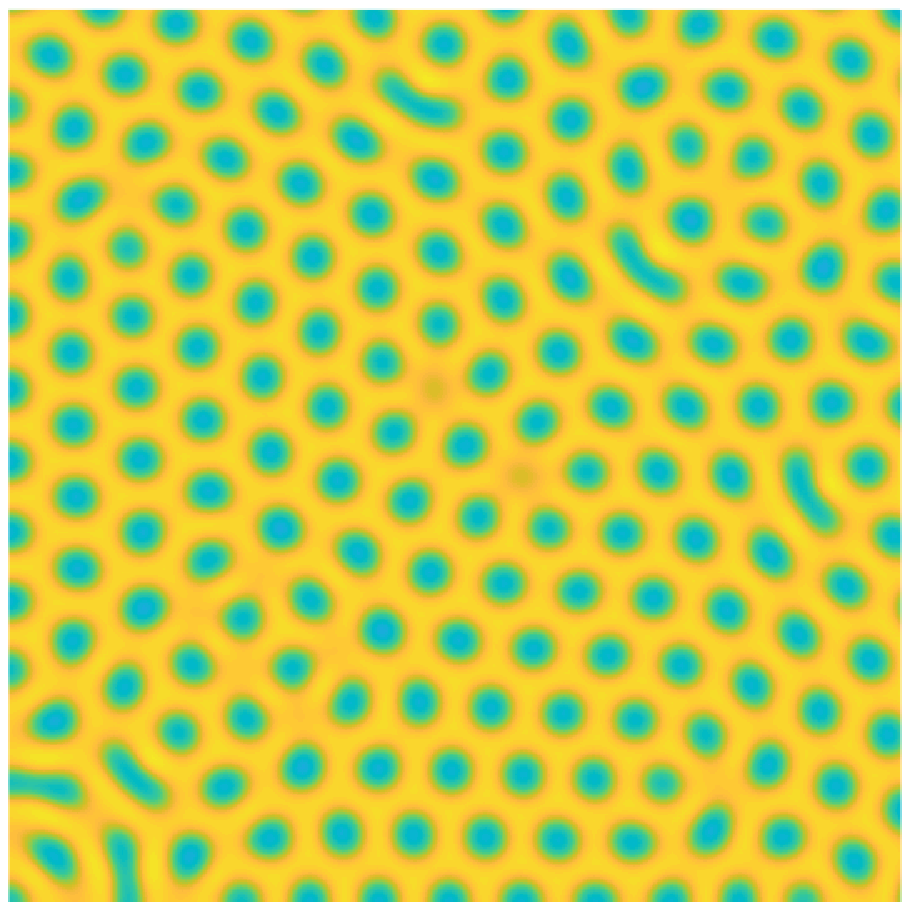}}
\subfigure[$\kappa=0.61,\,s=0.85$]{
\includegraphics[width=0.31\textwidth]{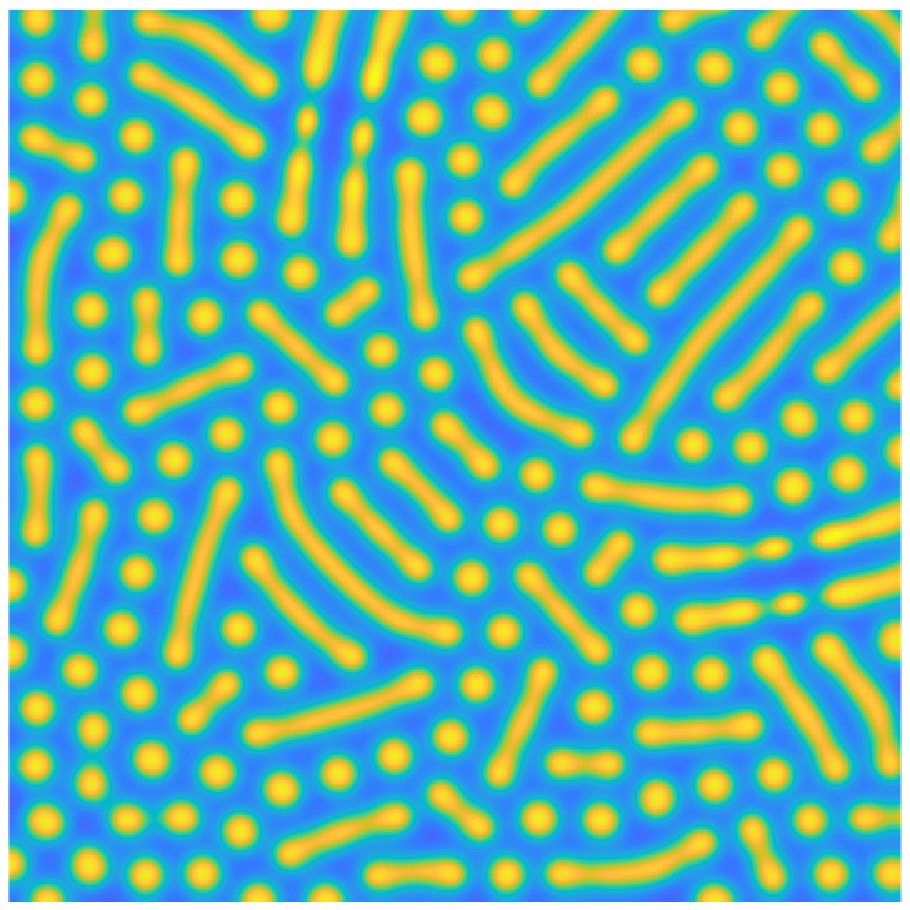}}
\subfigure[$\kappa=0.61,\,s=0.75$]{
\includegraphics[width=0.31\textwidth]{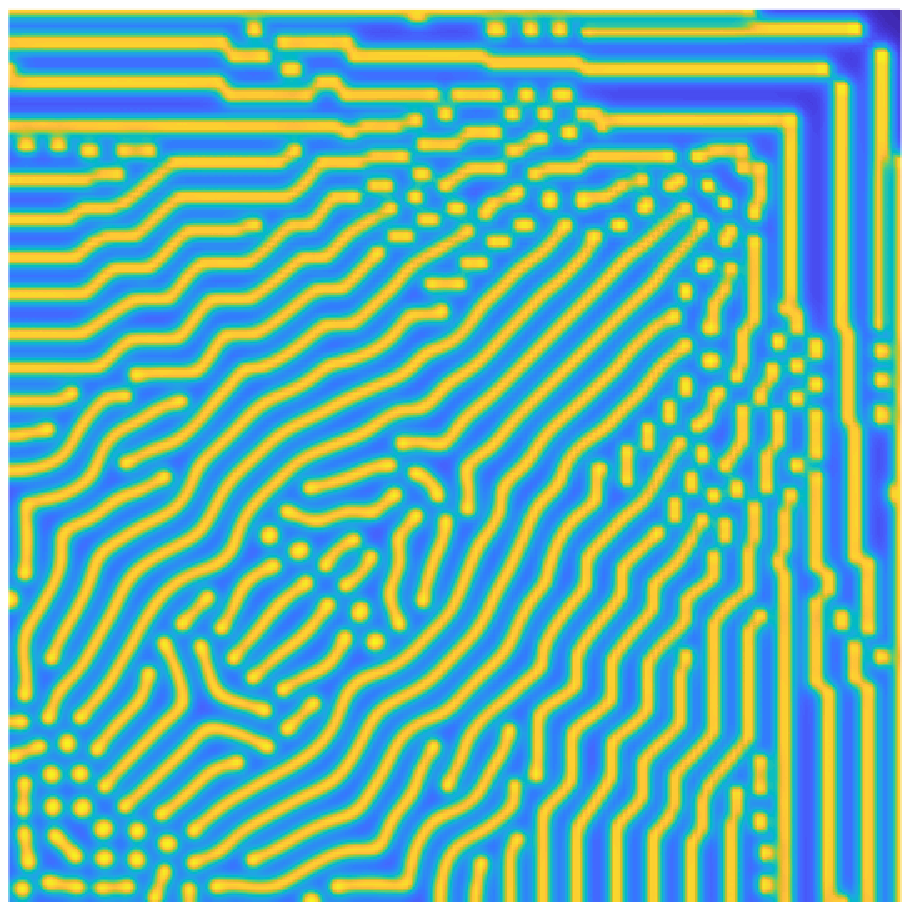}}
\caption{Pattern formation of the numerical approximation at final time $T=30000$ for different combinations of $\kappa$ and $s$.}
\label{gspattern}
\end{figure}
\subsection{FitzHugh-Nagumo equation}
The FitzHugh-Nagumo equation represents one of the simplest models for studying excited media. We study the following space-fractional equation:
\begin{equation}
\left\{\begin{array}{ll}
  \frac{\partial u}{\partial t}=-K_{u}(-\Delta)^{\alpha/2}u(x)+u(1-u)(u-\mu)-v,\vspace{1.2ex}\\
  \frac{\partial v}{\partial t}=\epsilon(\beta u-\gamma v-\delta),
  \end{array}\right.
\end{equation}
where $u$ is the excitation variable and $v$ the recovery variable. The FitzHugh-Nagumo equations have been used to qualitatively model many biological phenomena. It is known that when $\delta=0$, $(u^{*},v^{*})=(0,0)$ is a stable point for this problem. For our numerical experiments, we choose the initial condition
\begin{equation}
u_{0}=
\left\{\begin{array}{ll}
  1,&\quad (x,y)\in (-1,0)\times (-1,0),\\
  0,&\quad \text{elsewhere},
  \end{array}\right.
\end{equation}
\begin{equation}
v_{0}=
\left\{\begin{array}{ll}
  0.1,&\quad (x,y)\in (-1,1)\times (0,1),\\
  0,&\quad \text{elsewhere},
  \end{array}\right.
\end{equation}
which allows the initial condition to rotate clockwise and generate spiral waves. Other parameters are taken as $\mu=0.1$, $\epsilon=0.01$, $\beta=0.5$, $\gamma=1$, $\delta=0$ for the same purpose. Here we choose $N=800$, $\tau=0.1$ for numerical computation.

Numerical results for different combinations of $s$ and $K_{u}$ are displayed in Figs. \ref{fhnnum14}-\ref{fhnnum15}. In order to display the results more clearly, we use in Figs. \ref{fhnnum14} and \ref{fhnnum55} a bigger domain $[-5, 5]^2$, while for smaller $K_{u}$ case as in Fig. \ref{fhnnum15} a smaller window of $[-3, 3]^2$ is taken. By comparison within each figure, we can observe that the width of excitation wavefront is reduced for smaller values of $s$. By comparison among Fig.\ref{fhnnum14}, \ref{fhnnum55} and \ref{fhnnum15} for same values of $s$, it is obvious that in unbounded domains, the affected area is decreased with smaller diffusion coefficient $K_{u}$.

 \begin{figure}[!htbp]
\subfigure{
\centering
{\hspace*{-1.2cm}}{\includegraphics[width=1\textwidth]{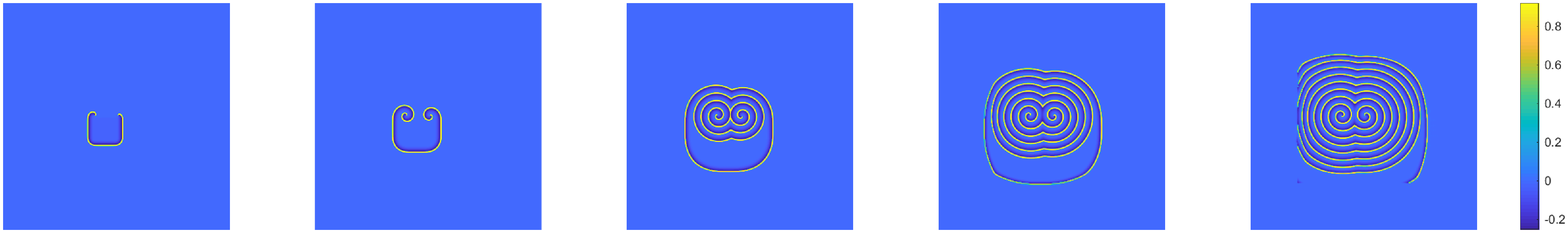}}}
\vskip -5pt
\subfigure{
\centering
{\hspace*{-1.2cm}}{\includegraphics[width=1\textwidth]{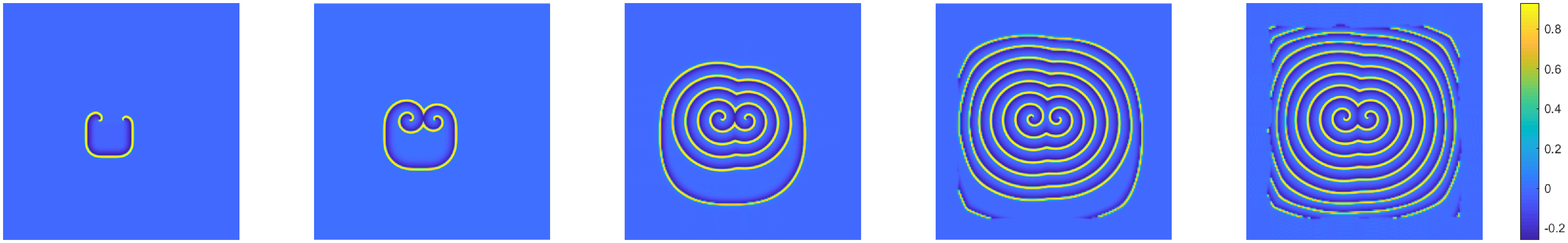}}}
\vskip -5pt
\subfigure{
{\hspace*{-1.2cm}}{\includegraphics[width=1\textwidth]{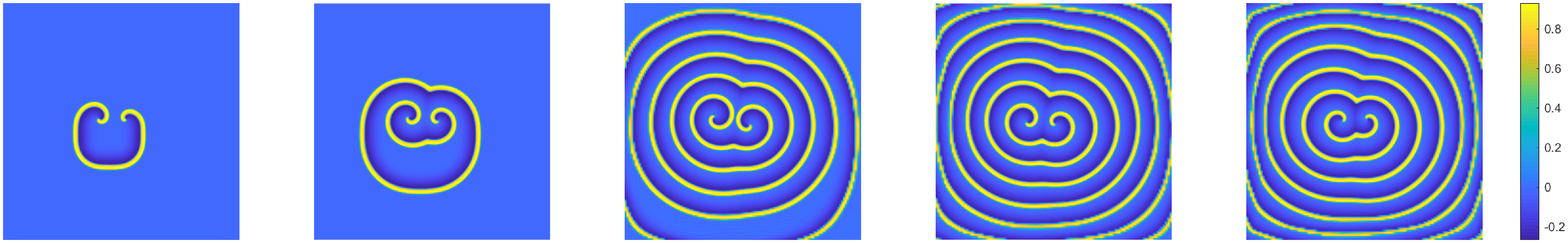}}}
\caption{Numerical approximations of FitzHugh-Nagumo equation with $K_{u}=1\times 10^{-4}$ with different values of $s$. Top: $s=0.75$, middle: $s=0.85$, bottom: $s=1$. Time from left to right: $t=200,400,1000,1500,2000$.}\label{fhnnum14}
\end{figure}

\begin{figure}[!htbp]
\subfigure{
\centering
{\hspace*{-1.2cm}}{\includegraphics[width=1\textwidth]{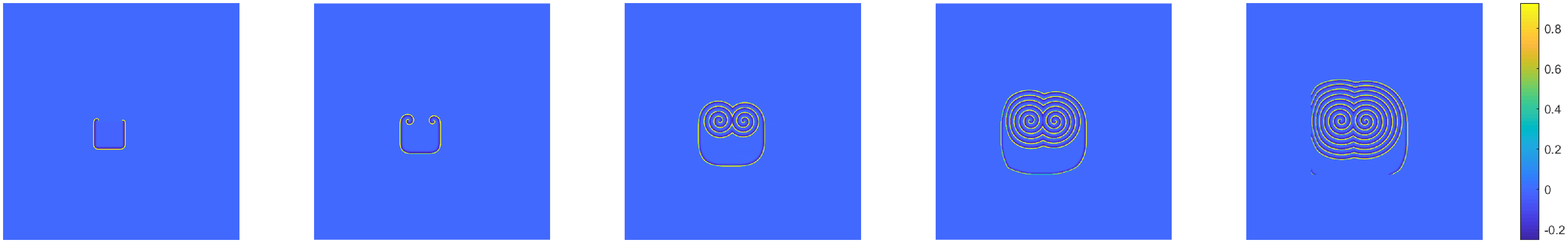}}}
\vskip -5pt
\subfigure{
\centering
{\hspace*{-1.2cm}}{\includegraphics[width=1\textwidth]{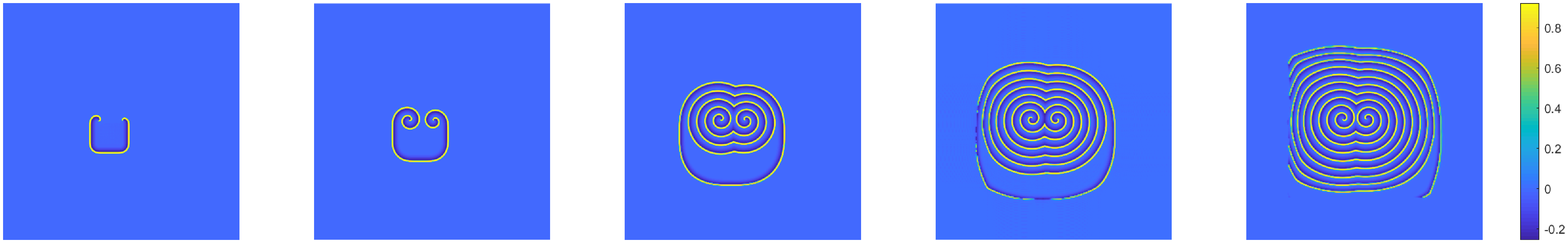}}}
\vskip -5pt
\subfigure{
{\hspace*{-1.2cm}}{\includegraphics[width=1\textwidth]{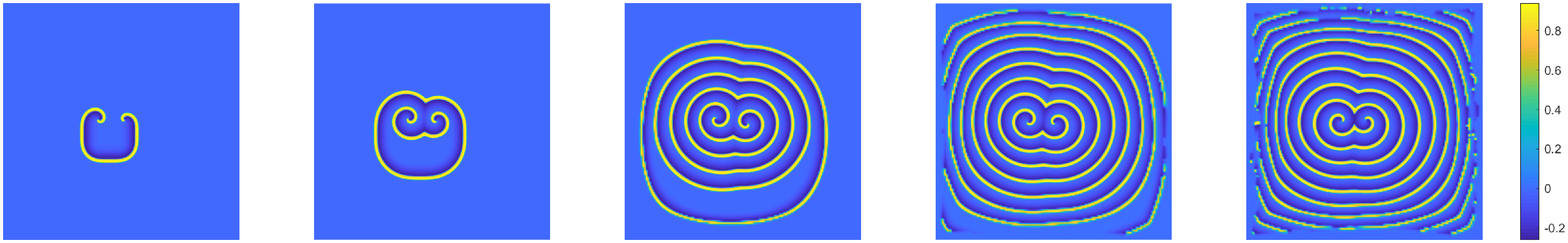}}}
\caption{Numerical approximations of FitzHugh-Nagumo equation with $K_{u}=5\times 10^{-5}$ with different values of $s$. Top: $s=0.75$, middle: $s=0.85$, bottom: $s=1$. Time from left to right: $t=200,400,1000,1500,2000$.}\label{fhnnum55}
\end{figure}

 \begin{figure}[!htbp]
\subfigure{
\centering
{\hspace*{-1.2cm}}{\includegraphics[width=1\textwidth]{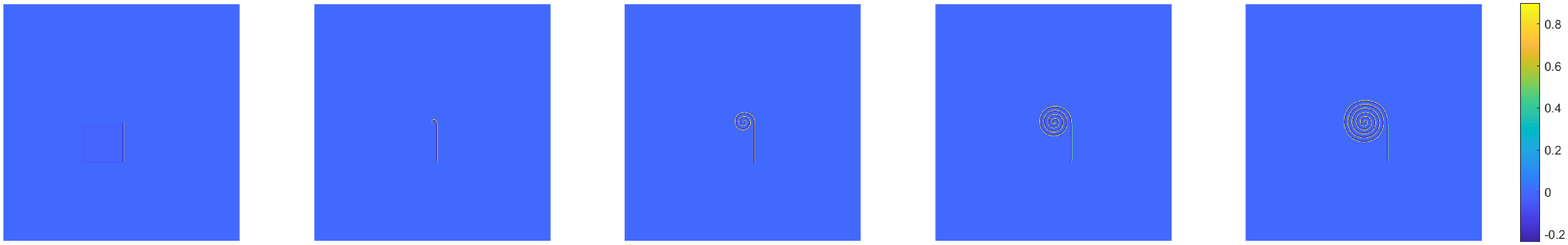}}}
\vskip -5pt
\subfigure{
\centering
{\hspace*{-1.2cm}}{\includegraphics[width=1\textwidth]{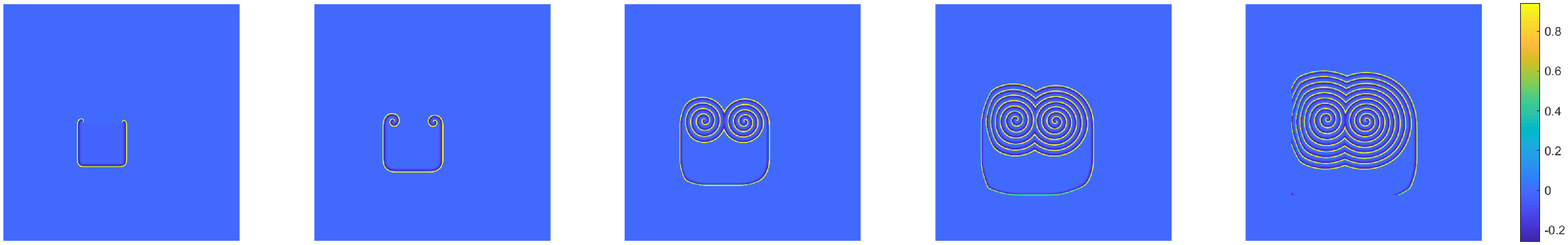}}}
\vskip -5pt
\subfigure{
{\hspace*{-1.2cm}}{\includegraphics[width=1\textwidth]{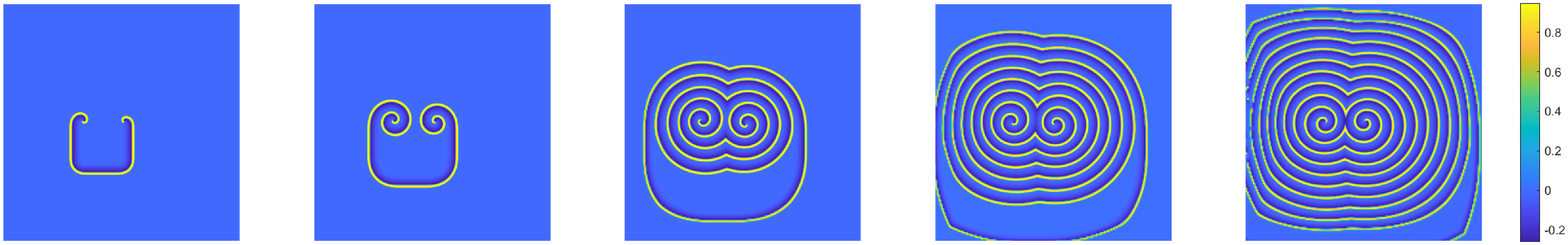}}}
\caption{Numerical approximations of FitzHugh-Nagumo equation with $K_{u}=1\times 10^{-5}$ with different values of $s$. Top: $s=0.75$, middle: $s=0.85$, bottom: $s=1$. Time from left to right: $t=200,400,1000,1500,2000$.}\label{fhnnum15}
\end{figure}

\subsection{Allen-Cahn equation in three-dimensional space}
We consider the spectral-Galerkin method for the fractional-in-space Allen-Cahn equation in three spatial dimensions. Here, the reaction term is given in \eqref{Fu}. In our numerical experiment, we set $N_{x}=N_{y}=N_{z}=200$, other parameters are the same as the example for 2D problem in subsection \ref{2dalen}. With the random initial data
\begin{equation}
u_{0}=0.5+0.1(\text{rand}-0.5),
\end{equation}
 we show the numerical results for isosurfaces of $u(x,y,z)=0.5$ at final time $T=100$ in Fig.\,\ref{allen3drand} with different $s$. Although the solution domain is infinity, we display in Fig.\,\ref{allen3drand} the numerical results in $[-1,1]^3$. It is observed from Fig.\,\ref{allen3drand} that the chaotic structures are transformed to large bulk regions as time evolves. Meanwhile, fractional Allen-Cahn models with larger values of $s$ produce more visible spatial structures, similar to the results for 2D.

\begin{figure}[!htbp]
\centering
\subfigure[$s=0.3$]{
\includegraphics[width=0.48\textwidth]{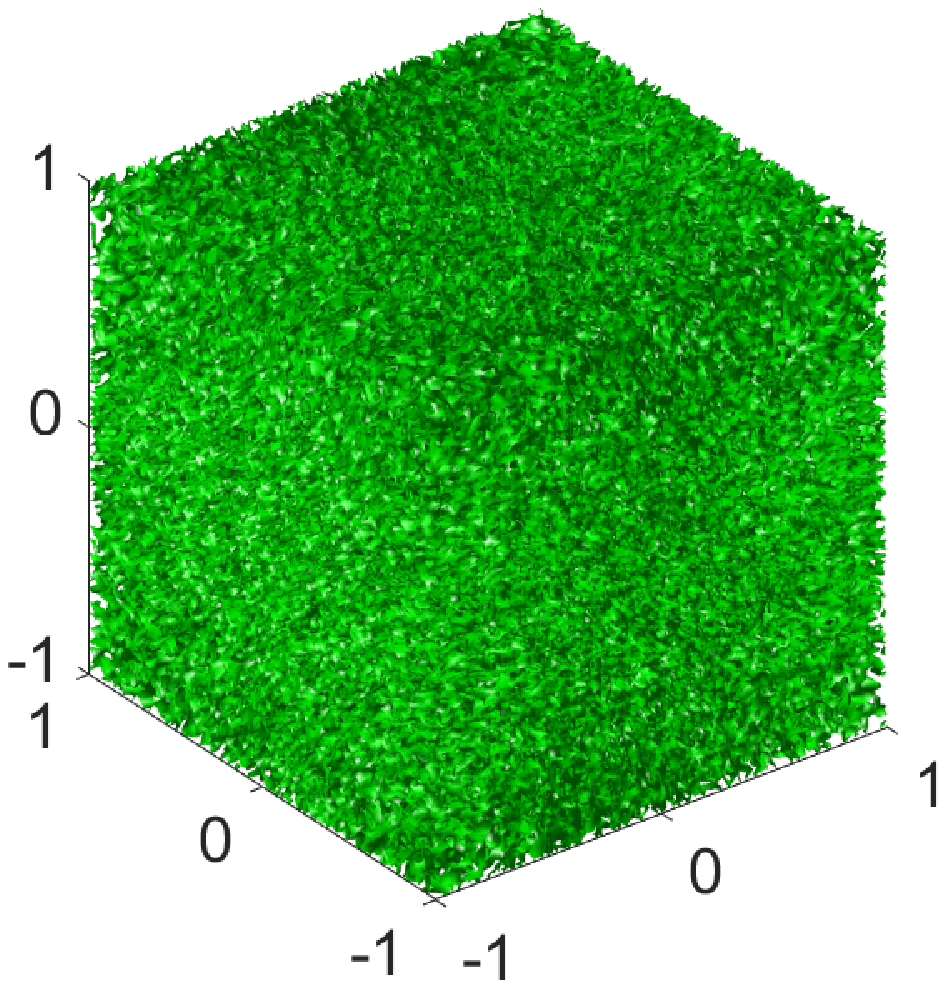}}
\subfigure[$s=0.5$]{
\includegraphics[width=0.48\textwidth]{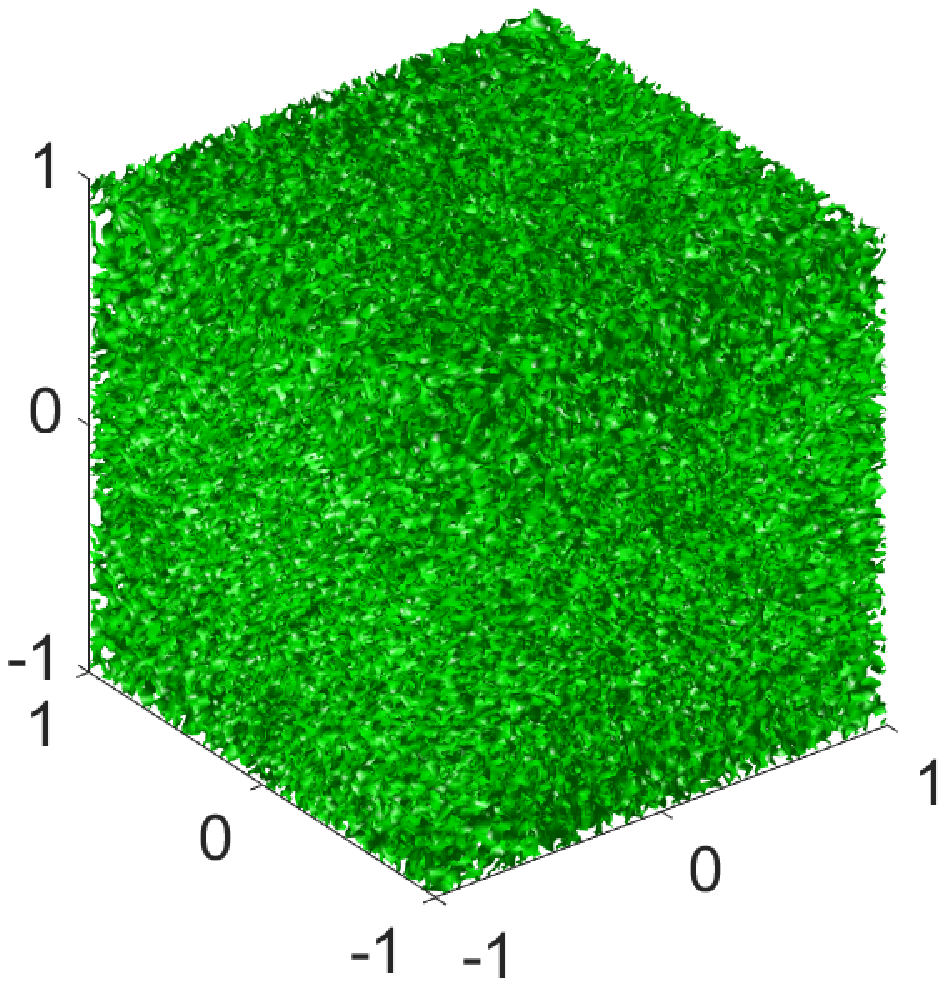}}
\subfigure[$s=0.6$.]{
\includegraphics[width=0.48\textwidth]{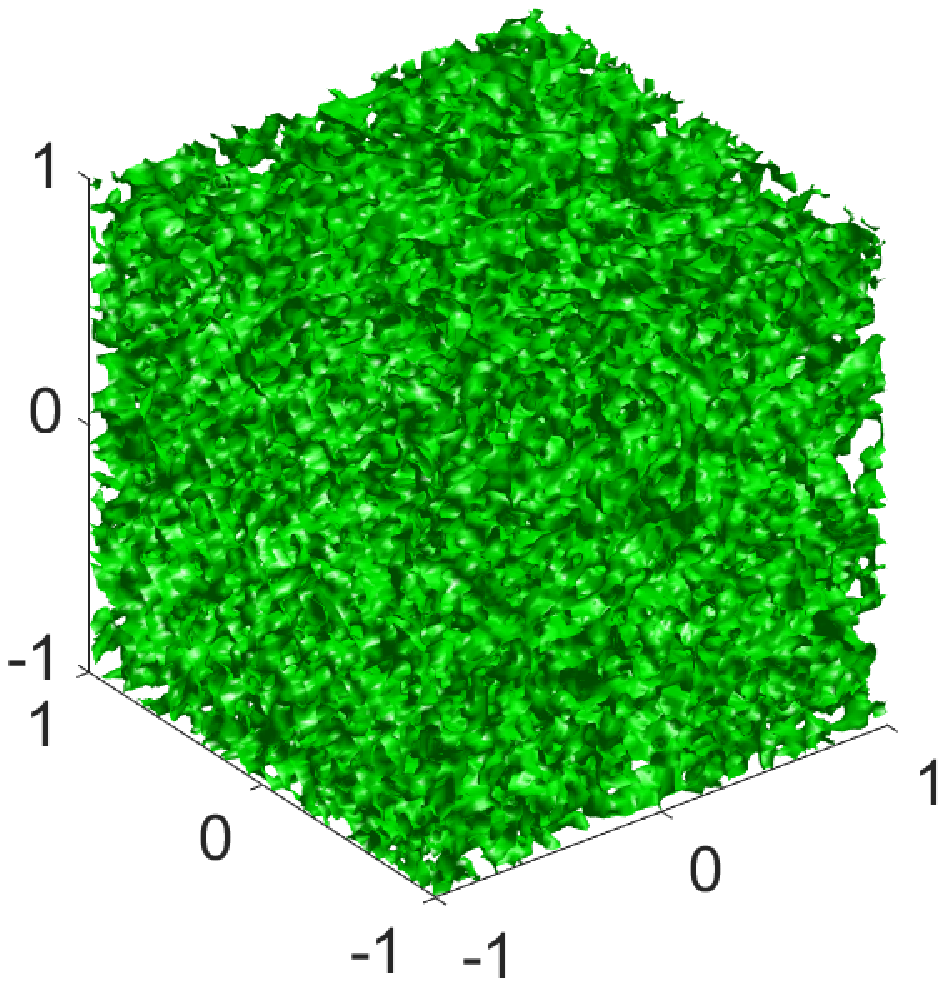}}
\subfigure[$s=0.8$.]{
\includegraphics[width=0.48\textwidth]{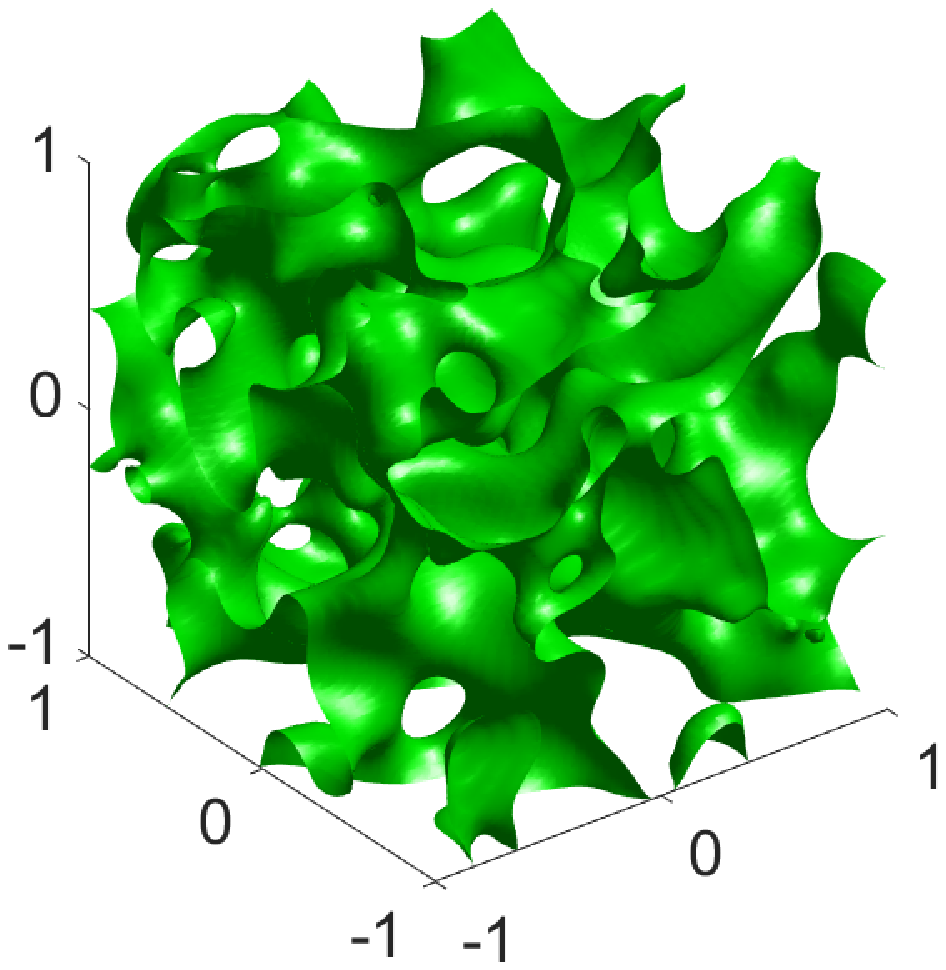}}
\subfigure[$s=1$]{
\includegraphics[width=0.48\textwidth]{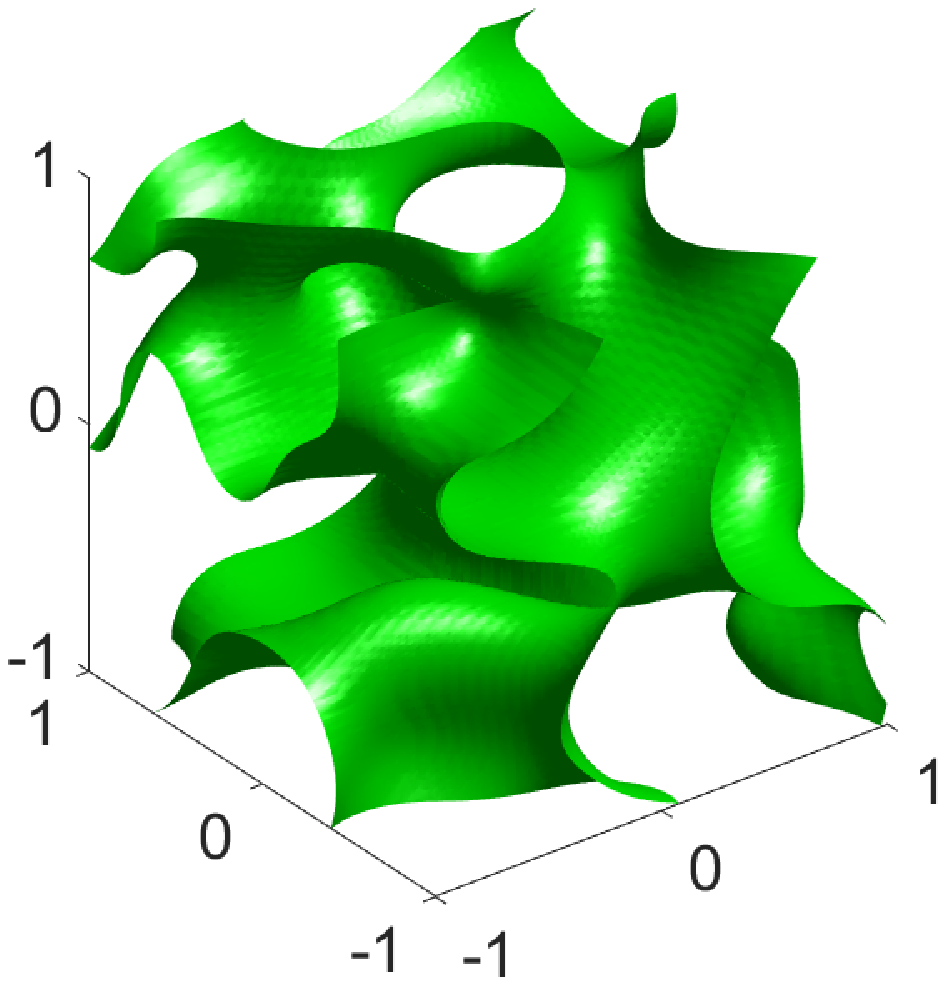}}
\subfigure[$s=1.2$.]{
\includegraphics[width=0.48\textwidth]{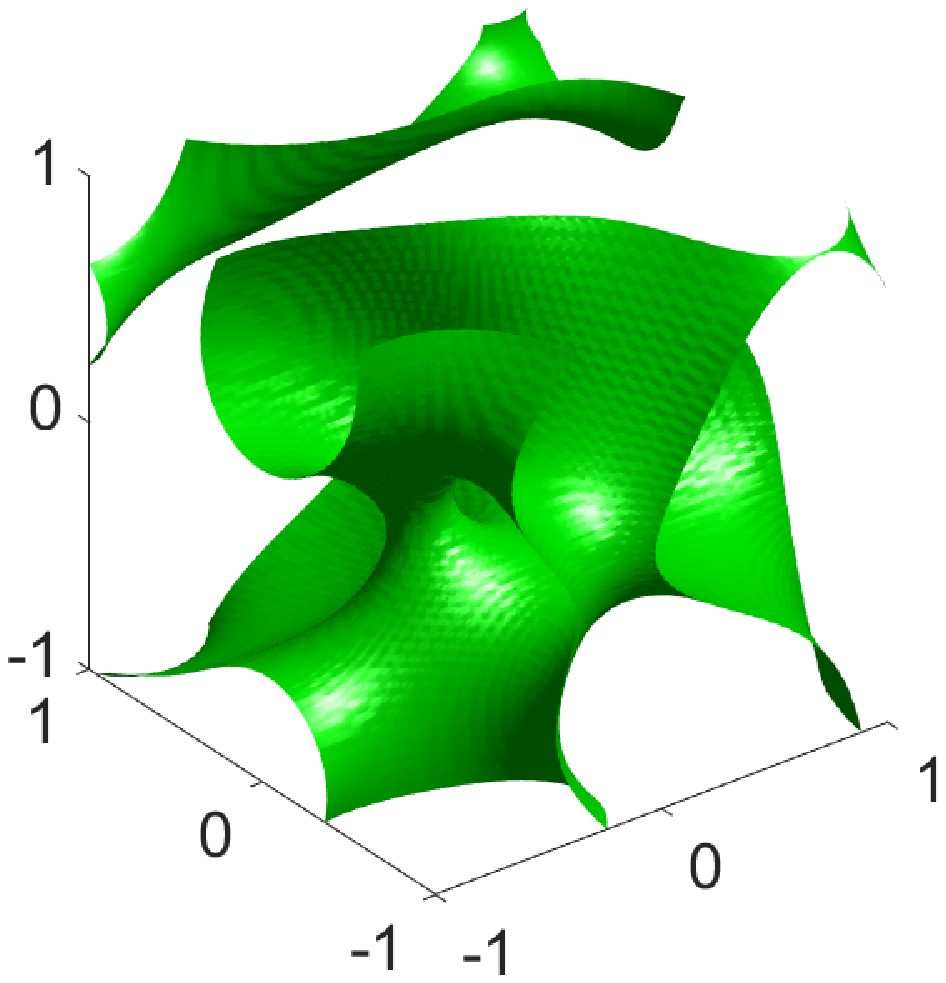}}
\caption{Isosurfaces of space-fractional Allen-Cahn equation in three-dimension where $u(x,y,z)=0.5$ associated with different $s$ at final time $T=100$.}
\label{allen3drand}
\end{figure}

\section{Concluding remarks}

In this work, we applied the efficient and accurate Fourier-like method for fractional reaction-diffusion equations in unbounded domains as proposed in \cite{sheng2020fast}. The method is particularly effective for problems involving the fractional Laplacian. To fully discretize the underlying systems, we propose to use an accurate time marching scheme based on ETDRK4. The main advantage of our method is that it yields a fully diagonal representation of the fractional Laplace operator, which provides great efficiency. Numerical examples are presented to illustrate the effectiveness of the proposed method. We also remark that this method can also be used to deal with problems involving Riesz derivatives in different directions and coupled systems of arbitrary $n$ species.

There are still several possible extensions following the present method, e.g., space- and time-fractional problems involving the fractional Laplacian issues require future investigations. Other types of
fractional derivatives such as  Riemann-Liouville fractional derivatives and Caputo fractional derivatives will be the topic of our forthcoming research.

\section{Acknowledgement}
The work of the author is partially supported by
the NSF of China (under the Grant No. 11731006).

\bibliographystyle{siam}
\bibliography{ref,rational}
\end{document}